\UseRawInputEncoding

\documentclass[12pt, reqno]{amsart}
\usepackage{amssymb}
\usepackage{eucal}
\usepackage{amsmath}
\usepackage{amscd}
\usepackage[dvips]{color}
\usepackage{mathtools}
\usepackage{multicol}
\usepackage[all]{xy}           %xypic macro for latex2.09
\usepackage{graphicx}
\usepackage{color}
\usepackage{colordvi}
\usepackage{xspace}
\usepackage{ulem}
\usepackage{cancel}
\usepackage{tikz}
\usepackage{longtable}
\usepackage{multirow}
\usepackage{ifpdf}
\ifpdf
\usepackage[colorlinks,final,backref=page,hyperindex]{hyperref}
\else
\usepackage[colorlinks,final,backref=page,hyperindex,hypertex]{hyperref}
\fi

\begin{document}
	\newcommand {\emptycomment}[1]{} %to remove paragraphs
	
	\newcommand{\tabincell}[2]{\begin{tabular}{@{}#1@{}}#2\end{tabular}}

	\newcommand{\nc}{\newcommand}
	\newcommand{\delete}[1]{}

	%%%%%%Use the next black to suppress labels
	%\delete{
		\nc{\mlabel}[1]{\label{#1}}  % Use this to suppress names
		\nc{\mcite}[1]{\cite{#1}}  % Use this to suppress names
		\nc{\mref}[1]{\ref{#1}}  % Use this to suppress names
		\nc{\meqref}[1]{Eq.~\eqref{#1}} % Use this to suppress names
		\nc{\mbibitem}[1]{\bibitem{#1}} % Use this to show number
		%}
	
	%%%%%%%%%%%Use the next block to show labels
	\delete{
		\nc{\mlabel}[1]{\label{#1}  % Use the next two lines to show names
			{\hfill \hspace{1cm}{\bf{{\ }\hfill(#1)}}}}
		\nc{\mcite}[1]{\cite{#1}{{\bf{{\ }(#1)}}}}  % Use this lines to show names
		\nc{\mref}[1]{\ref{#1}{{\bf{{\ }(#1)}}}}  % Use this lines to show names
		\nc{\meqref}[1]{Eq.~\eqref{#1}{{\bf{{\ }(#1)}}}} % Use this lines to show names
		\nc{\mbibitem}[1]{\bibitem[\bf #1]{#1}} % Use this to show name
	}

	%%%%%%%%%%%%%%%%%%%%%%%% Statements
	\newtheorem{thm}{Theorem}[section]
	\newtheorem{lem}[thm]{Lemma}
	\newtheorem{cor}[thm]{Corollary}
	\newtheorem{pro}[thm]{Proposition}
	\newtheorem{conj}[thm]{Conjecture}
	\theoremstyle{definition}
	\newtheorem{defi}[thm]{Definition}
	\newtheorem{ex}[thm]{Example}
	\newtheorem{rmk}[thm]{Remark}
	\newtheorem{pdef}[thm]{Proposition-Definition}
	\newtheorem{condition}[thm]{Condition}
	\newtheorem{question}[thm]{Question}

	\renewcommand{\labelenumi}{{\rm(\alph{enumi})}}
	\renewcommand{\theenumi}{\alph{enumi}}
	\renewcommand{\labelenumii}{{\rm(\roman{enumii})}}
	\renewcommand{\theenumii}{\roman{enumii}}

	\nc{\tred}[1]{\textcolor{red}{#1}}
	\nc{\tblue}[1]{\textcolor{blue}{#1}}
	\nc{\tgreen}[1]{\textcolor{green}{#1}}
	\nc{\tpurple}[1]{\textcolor{purple}{#1}}
	\nc{\btred}[1]{\textcolor{red}{\bf #1}}
	\nc{\btblue}[1]{\textcolor{blue}{\bf #1}}
	\nc{\btgreen}[1]{\textcolor{green}{\bf #1}}
	\nc{\btpurple}[1]{\textcolor{purple}{\bf #1}}

	%new commands
	
	\newcommand{\End}{\text{End}}

	\nc{\calb}{\mathcal{B}}
	\nc{\call}{\mathcal{L}}
	\nc{\calo}{\mathcal{O}}
	\nc{\frakg}{\mathfrak{g}}
	\nc{\frakh}{\mathfrak{h}}
	\nc{\ad}{\mathrm{ad}}
	\def \gl {\mathfrak{gl}}
	\def \g {\mathfrak{g}}

	\nc{\ccred}[1]{\tred{\textcircled{#1}}}
	
	%\nc{\move}[1]{\footnote{#1}}
	
	\newcommand{\cm}[1]{\textcolor{purple}{\underline{CM:}#1 }}

	\newcommand\blfootnote[1]{%
		\begingroup
		\renewcommand\thefootnote{}\footnote{#1}%
		\addtocounter{footnote}{-1}%
		\endgroup
	}

	%%%%%%%%%%%%%%%%%%%%%%%%%%%%%%%%%%%%%%%%%%%%%%%%%%%%%%%%%%%%%%%%%%
	
	%%%%%%%%%%%%%%%%%%%%%%%%%%%%%%%%%%%%%%%%%%%%%%%%%%%%%%%%%%%%%%%%%%%%%%%%%%%
	%%%%%%%%%%%%%%%%%%%%%%    Title    %%%%%%%%%%%%%%%%%%%%%%%%%%%%%%%%%%%%%%%%
	\title[Matched pairs and double construction bialgebras]
	{Matched pairs and double construction bialgebras of (transposed) Poisson 3-Lie algebras}

	\author{Kecheng Zhou}
	\address{
		School of Mathematical Sciences  \\
		Zhejiang Normal University\\
		Jinhua 321004 \\
		China}
	\email{zhoukc@zjnu.edu.cn}

	\author{Chuangchuang Kang}
	\address{
		School of Mathematical Sciences  \\
		Zhejiang Normal University\\
		Jinhua 321004 \\
		China}
	\email{kangcc@zjnu.edu.cn}
	\author{Jiafeng L\"u}
	\address{
		School of Mathematical Sciences  \\
		Zhejiang Normal University\\
		Jinhua 321004 \\
		China}
	\email{jiafenglv@zjnu.edu.cn}
	
	\blfootnote{*Corresponding Author:\ Chuangchuang\ Kang. Email:\ kangcc@zjnu.edu.cn\ }

	\begin{abstract}
		
		Double construction bialgebras for Poisson 3-Lie algebras and transposed Poisson 3-Lie algebras are defined and studied using matched pairs. Poisson 3-Lie algebras and transposed Poisson 3-Lie algebras are constructed on direct sums and tensor products of vector spaces. Matched pairs, Manin triples, and double construction Poisson 3-Lie bialgebras are shown to be equivalent.
		Since the double construction approach does not apply to bialgebra theory for transposed Poisson 3-Lie algebras, admissible transposed Poisson 3-Lie algebras are introduced. These algebras are both transposed Poisson 3-Lie algebras and Poisson 3-Lie algebras. An equivalence between matched pairs and double construction admissible transposed Poisson 3-Lie bialgebras is established.
		
		%This paper defines and studies double construction bialgebras for Poisson 3-Lie algebras and transposed Poisson 3-Lie algebras using their matched pairs. First, we construct Poisson 3-Lie algebras and transposed Poisson 3-Lie algebras on direct sums and tensor products of vector spaces. Second, we prove the equivalence among matched pairs, Manin triples, and double construction Poisson 3-Lie bialgebras. However, non-trivial bilinear forms analogous to those in Poisson bialgebras do not exist for transposed Poisson 3-Lie algebras. To address this, we introduce admissible transposed Poisson 3-Lie algebras, which are both transposed Poisson 3-Lie algebras and Poisson 3-Lie algebras. We establish equivalence between matched pairs and double construction admissible transposed Poisson 3-Lie bialgebras.
	\end{abstract}

	\subjclass[2020]{
		17A40, %Ternary compositions
		16T10,%Bialgebras
		17A30, %Algebras satisfying other identities
		17A42, %Other n-ary compositions (n ≥ 3)
		17B10, %Representations of Lie algebras and Lie superalgebras, algebraic theory (weights)
		17B63. %Poisson algebra
		%17B62. %Lie bialgebras; Lie coalgebras
	}

	\keywords{ Poisson 3-Lie algebras, transposed Poisson 3-Lie algebras, representations, matched pairs, manin triples, double construction bialgebras}
	
	\maketitle
	
	%\vspace{-1.5cm}
	
	\tableofcontents
	
	\allowdisplaybreaks
	
	\section{Introduction}
	The aim of this paper is to construct matched pairs and double construction bialgebras for Poisson 3-Lie algebras and transposed Poisson 3-Lie algebras, motivated by the question \cite[Question 35]{Beites-1}: define and study transposed Poisson bialgebras.
	\subsection{ (Transposed) Poisson algebras}
	Poisson algebras originated from the study of Hamiltonian mechanics and are connected to various fields, including   Poisson manifolds \cite{K.Bhaskara}, quantum groups \cite{Chari}, operads \cite{Goze}, quantization theory \cite{Kontsevich}, quantum mechanics \cite{Odzijewicz} and algebraic geometry \cite{Polishchuk}. 
	\begin{defi}
		A Poisson algebra is a triple $(A,\cdot,[\cdot,\cdot])$, where  $(A,\cdot)$ is a commutative associative algebra and $(A, [\cdot,\cdot ])$ is a Lie algebra satisfying the Leibniz rule:  
		\begin{equation*}\label{eq:poisson-alg-1}  
			[x\cdot y,z]=x\cdot[ y,z]+ y\cdot[x,z], \quad \forall~ x,y,z\in A.
		\end{equation*}  
	\end{defi}
	Bai, Bai, Guo, and Wu \cite{Bai2020} introduced the concept of a transposed Poisson algebra. This is achieved by exchanging the roles of the two binary operations in the Leibniz rule, resulting in the following transposed Leibniz rule:  
	\begin{equation*}  
		2x \cdot [y, z] = [x \cdot y, z] + [y, x \cdot z], \quad \forall~ x, y, z \in A.
	\end{equation*}
	The transposed Poisson algebra shares some properties with the Poisson algebra, such as  being constructed from a commutative algebra with two commuting derivations. However, they differ significantly. For example, the Leibniz rule implies that adjoint operators of the Lie algebra are derivations of the commutative associative algebra, while the transposed Leibniz rule shows that left multiplication in the commutative associative algebra is a $\frac{1}{2}$-derivation of the Lie algebra \cite{Ferreira}. 
	
	In recent years, transposed Poisson algebras have attracted significant attention.
	Beites, Ouaridi and Kaygorodov classified all complex 3-dimensional transposed Poisson algebras \cite{Beites}.
	Fern\'andez Ouaridi proved that a transposed Poisson algebra is simple if and only if its associated Lie bracket is simple, with  analogous results for transposed Poisson superalgebras \cite{Ouaridi}.
	Using the connection with the $\frac{1}{2}$-derivations, Kaygorodov and Khrypchenko described transposed Poisson structures on block Lie algebras and superalgebras\cite{Kaygorodov-2}, generalized Witt algebras \cite{Kaygorodov-3}, and Witt type algebras \cite{Kaygorodov}. Such structures  were also studied on complex Galilean type Lie algebras and superalgebras \cite{Kaygorodov-4}, not-finitely graded Witt-type algebras\cite{Kaygorodov-1}, and Virasoro-type algebras \cite{Kaygorodov-5}. For more results and open questions, see \cite{Beites-1}.  
	
	\subsection{ (Transposed) Poisson $n$-Lie algebras}
	The notion of Poisson $3$-Lie algebras, or more generally,
	Poisson $n$-Lie algebras (also called Nambu-Poisson algebras \cite{L.Takhtajan}) can be seen as a generalization of Poisson algebras to higher arities.
	These structures arise from the study of Nambu mechanics \cite{Y.Nambu} and are applied in Nambu-Poisson manifolds \cite{L.Takhtajan}. See \cite{Azcarraga}  for more details about Poisson $n$-Lie algebras and applications in mathematical physics.
	\begin{defi}
		An $n$-Lie algebra \cite{Filippov} is a vector space $A$ equipped with an $n$-ary bracket $[\cdot, \cdots, \cdot]:\wedge^n A \rightarrow A$, such that for all $x_1, \cdots, x_{n-1}, y_1, \cdots, y_n \in A$, the following Filippov-Jacobi identity holds:  
		\begin{equation*}  
			[x_1, \cdots, x_{n-1}, [y_1, \cdots, y_n]] = \sum_{i=1}^n [y_1, \cdots, y_{i-1}, [x_1, \cdots, x_{n-1}, y_i], y_{i+1}, \cdots, y_n].
		\end{equation*}  
	\end{defi}  
	\begin{defi}
		A Poisson $n$-Lie algebra  \cite{A. Makhlouf} is both a commutative associative
		algebra $(A, \cdot)$ and an $n$-Lie algebra $(A,[\cdot,\cdots,\cdot])$ satisfying the generalized Leibniz rule:
		\begin{equation*}
			[w \cdot x_1,x_2, \cdots,x_n] = w\cdot[x_1,x_2,\cdots,x_n]+x_1\cdot[w,x_2,\cdots,x_n], \quad \forall~x_1, \cdots, x_n, w\in A.
		\end{equation*}
	\end{defi} 
	Transposed Poisson $n$-Lie algebras, introduced in \cite{Bai2020}, generalize transposed Poisson algebras to the $n$-ary case and are the dual notion of Poisson $n$-Lie algebras.
	\begin{defi}
		A transposed Poisson $n$-Lie algebra $(B,\cdot,[\cdot,\cdots,\cdot])$ is both a commutative associative algebra $(B, \cdot)$ and an $n$-Lie algebra $(B,[\cdot,\cdots,\cdot])$ satisfying:
		\begin{equation*}
			n w\cdot[x_1,x_2, \cdots, x_n] = \sum_{i=1}^n[x_1,\cdots, w\cdot x_i,\cdots,x_n],\quad \forall ~ x_1, \cdots, x_n, w\in B.
		\end{equation*}
	\end{defi}
	With an additional ``strongness'' condition, a Poisson $n$-Lie algebra with a derivation can construct an $(n + 1)$-Lie algebra \cite{Dzhumadildav}. Similarly, in \cite{Bai2020}, Bai, Bai, Guo, and Wu proved that each transposed Poisson algebra defines a transposed Poisson $3$-Lie algebra via the multiplication:  
	$$
	[x,y,z]=\mathfrak{D}(x)\cdot[y,z]-\mathfrak{D}(y)\cdot[x,z]+\mathfrak{D}(z)\cdot[x,y], \quad\forall~ x,y,z\in A.
	$$
	Under certain ``strong'' conditions, the $n$-ary ($n\geq 3$) conjecture (from \cite[Question 37]{Beites-1}) of this construction  was solved in \cite{Jiang}. The concept of $\frac{1}{n}$-derivations for $n$-ary algebras, as a special case of $\delta$-derivations, was introduced by Kaygorodov in \cite{Kaygorodov-6}. Using $\frac{1}{n}$-derivations, Ferreira, Kaygorodov, and Lopatkin showed that no non-trivial transposed Poisson $n$-Lie algebra (or superalgebra) structures exist on complex simple finite-dimensional $n$-Lie algebras (or superalgebras) \cite{Ferreira}. Additionally, a classification of transposed Poisson 3-Lie algebras of dimension $3$ is given in \cite{Jiang}, which gives examples of transposed Poisson 3-Lie algebras and partially solves the problem \cite[Question 38]{Beites-1}.
	
	\subsection{Matched pairs,  Manin triples, and bialgebras of associative algebras}
	In 1990, Majid introduced the notion of matched pairs of Lie groups \cite{Majid}. These can form bicrossproduct groups, Hopf algebras, Hopf-von Neumann algebras, or Kac algebras through cross product and coproduct. 
	The idea has been generalized to various contexts, including matched pairs of groups related to braided groups and solutions of the Yang-Baxter equation \cite{Etingof}, matched pairs of Leibniz algebras for structure extension \cite{Agore}, and matched pairs of Lie groupoids and Lie algebroids in geometry \cite{Mackenzie, Mokri}. The notion of a matched pair of Lie algebras (also called double Lie algebras \cite{Lu}) was introduced in the study of Lie bialgebras and Poisson-Lie groups, which are equivalent to Manin triples of Lie algebras and related to quantum groups \cite{Chari}.
	
	Motivated by the study of Lie bialgebras, the
	concepts of matched pairs,  double constructions, and
	Manin triples for associative algebras were systematically developed \cite{Bai2010}.
	A matched pair of associative algebras $A$ and $B$ include a bimodule $(l_A, r_A)$ of $A$ and a bimodule $(l_B, r_B)$ of $B$, and satisfy certain compatibility conditions (see Definition \ref{def: MP-CAA} for the commutative case). This induces an associative algebra structure on the direct sum $A\oplus B$.
	A Manin triple of associative algebra $A$ is a Frobenius algebra $(A, \mathcal{B})$, where $A=A_1\oplus A_1^*$, $A_1$ and $A_1^*$ are associative subalgebras of $A$, and $\mathcal{B}$ is the natural nondegenerate symmetric invariant bilinear form, which is also called double construction of Frobenius algebra \cite{Bai2010}. 
	Moreover, a double construction of Frobenius algebra is equivalent to an associative analogue for Lie bialgebras (called antisymmetric infinitesimal bialgebras). The relationships among these mathematical structures can be summarized in the following diagram:
	%\vspace{-.1cm}
	\begin{equation*}
		\begin{split}
			\xymatrix{
				\text{Matched pairs of }\atop \text{associative algebras} 
				\ar@2{<->}[r]& \text{Antisymmetric infinitesimal}\atop \text{ bialgebras}
				\ar@2{<->}[r] & \text{Double constructions of}\atop \text{Frobenius algebras}. }
		\end{split}
	\end{equation*}
	
	\subsection{Motivations}
	A Poisson bialgebra is both a double construction commutative Frobenius algebras and a Lie bialgebra \cite{X. Ni}. It fits naturally into a framework for constructing compatible Poisson brackets in integrable systems and shares many properties with Lie bialgebras. For example, the related algebra structures have the following equivalent diagram:
	%\vspace{-.1cm}
	\begin{equation*}
		\begin{split}
			\xymatrix{
				\text{Matched pairs of}\atop \text{Poisson algebras} 
				\ar@2{<->}[r]& \text{Poisson}\atop \text{ bialgebras}
				\ar@2{<->}[r] & \text{Standard Manin triples of }\atop \text{Poisson algebras}. }
		\end{split}
	\end{equation*}
	In \cite[Question 35]{Beites}, Beites, Ferreira and Kaygorodov proposed the following question:
	\begin{question} {\rm(Chengming Bai)}
		Define and study transposed Poisson bialgebras. 
	\end{question}
	This question was successfully resolved in \cite{G. Liu}. Inspired by this, it is naturally to study the bialgebra theory for (transposed) Poisson 3-Lie  algebras.
	The notion of 3-Lie bialgebras was introduced in \cite{Bai2019}, which studies two types of 3-Lie bialgebras: local cocycles and double constructions. They are extensions of Lie bialgebra structures.
	For commutative associative algebras, the associative bialgebra and infinitesimal bialgebra \cite{M. Aguiar,S.A. Joni} are well-known bialgebra structures. 
	
	In this paper, we combine the theories of 3-Lie bialgebras and commutative and cocommutative infinitesimal bialgebras to develop bialgebras for Poisson 3-Lie algebras, specifically the double construction Poisson 3-Lie bialgebras (see Definition \ref{defi:D-3-bi}). This construction is both a double construction of 3-Lie bialgebras and a commutative cocommutative infinitesimal bialgebra, compatible in a specific sense. We prove that matched pairs of Poisson 3-Lie algebras, Manin triples of Poisson 3-Lie algebras and double construction Poisson 3-Lie bialgebras are equivalent (see Theorem \ref{thm:one-one-co}).
	
	We construct double construction Poisson 3-Lie bialgebras using Manin triples and invariant bilinear forms on both commutative associative algebras and 3-Lie algebras. However, if the mixed products of $\cdot$ and $[\cdot,\cdot,\cdot]$ in a transposed Poisson 3-Lie algebra are non-trivial, then there is no non-trivial such bilinear form (see Proposition \ref{pro:no-bi-form}). To address this, we introduce a special class of transposed Poisson 3-Lie algebras called admissible transposed Poisson 3-Lie algebras (see Definition \ref{ad-Tran}), which are both transposed Poisson 3-Lie algebras and Poisson 3-Lie algebras. A double construction admissible transposed Poisson 3-Lie bialgebra is also a double construction Poisson 3-Lie bialgebra. Additionally, we provide matched pairs of admissible transposed Poisson 3-Lie algebras that are simultaneously matched pairs of Poisson 3-Lie algebras and transposed Poisson 3-Lie algebras. These relationships are summarized in the following diagram:
	\begin{equation*}
		\begin{split}
			\xymatrix{
				&\text{Double construction }\atop \text{admissible transposed Poisson 3-Lie bialgebras} \ar@2{<->}[d]^{Thm~\ref{thm:adm-m-d}}\ar@2{->}[r]^{\qquad\quad Cor~\ref{cor:ad-d-bi}}
				&\text{Double construction }\atop \text{Poisson 3-Lie bialgebras}
				\ar@2{<->}[d]^{Thm~\ref{thm:one-one-co}}\\
				\text{Matched pairs of }\atop \text{transposed Poisson 3-Lie algebras}\ar@2{<-}[r]^{Cor~\ref{cor:adm-m}}
				&\text{Matched pairs of}\atop \text{admissible transposed Poisson 3-Lie algebras} 
				\ar@2{->}[r]^{\qquad \quad Cor~\ref{cor:adm-m}} &\text{Matched pairs of}\atop \text{Poisson 3-Lie algebras} \\
				&&\ar@2{<->}[u]_{Thm~\ref{Manin triple}} \text{Standard Manin triples of }\atop \text{Poisson 3-Lie algebras.} }
		\end{split}
	\end{equation*}
	
	\subsection{Outline of the paper} 
	The paper is organized as follows. In Section \ref{sec:pre}, we provide preliminaries on 3-Lie algebras, commutative associative algebras, Poisson 3-Lie algebras, and transposed Poisson 3-Lie algebras. We introduce the concept of an admissible transposed Poisson 3-Lie algebra and give a non-trivial 4-dimensional example (Example \ref{ex:poisson-and-tran-poisson}). In Section \ref{sec:properties}, we construct Poisson 3-Lie algebras and transposed Poisson 3-Lie algebras on direct sums and tensor products of vector spaces. In Section \ref{sec:double-Poisson}, we define representations and dual representations of Poisson 3-Lie algebras, construct semi-direct product Poisson 3-Lie algebras (Proposition \ref{pro:possion-semi-direct}), and discuss matched pairs and Manin triples. We also explore double construction Poisson 3-Lie bialgebras and establish equivalences between these structures (Theorem \ref{thm:one-one-co}). In Section \ref{sec:trans-Poisson}, we introduce representations and matched pairs for transposed Poisson 3-Lie algebras and prove that no natural bialgebra theory exists due to trivial mixed products of $\cdot$ and $[\cdot,\cdot,\cdot]$ (Proposition \ref{pro:no-bi-form}). In Section \ref{sec:ad-poisson}, we define representations, matched pairs, and double construction bialgebras for admissible transposed Poisson 3-Lie algebras. We show equivalence between matched pairs and double construction admissible transposed Poisson 3-Lie bialgebras (Theorem \ref{thm:adm-m-d}) and provide a nontrivial example (Example \ref{ex:no-bi}). In Section \ref{sec:7}, we propose future research questions.

	In this paper, the base field is taken to be $\mathbb{F}$ unless otherwise specified. This is the field over which we take all associative, 3-Lie, and (transposed) Poisson 3-Lie  algebras, vector spaces, linear maps, and tensor products, etc., and the characteristic of $\mathbb{F}$ is zero. Moreover, throughout this paper, all algebras and vector spaces are assumed to be finite-dimensional.	
	Let $V$ be a vector space, $V^*$ be the dual space of $V$. For each positive integer $k$, we identify the tensor product $\otimes^k V$ with the space of multi-linear
	maps from $\underbrace{V^*\times \cdots \times V^*}_{k-times}\rightarrow \mathbb{F}$, such that
	\begin{equation}\label{eq:action}
		\langle\xi_1\otimes \cdots\otimes \xi_k,v_1\otimes\cdots\otimes v_k\rangle=\langle\xi_1,v_1\rangle\cdots\langle\xi_k,v_k\rangle,\quad \forall ~\xi_1,\cdots,\xi_k\in V^*,v_1,\cdots,v_k\in V,
	\end{equation}
	where $\langle \xi_i,v_i\rangle=\xi_i(v_i)$, $1\leq i\leq k$.
	\section{Preliminaries}\label{sec:pre}
	In this section, we first recall some necessary definitions and results for 3-Lie algebras, Poisson 3-Lie algebras, and transposed Poisson 3-Lie algebras. Then we introduce the concept of an admissible transposed Poisson 3-Lie algebra and provide a non-trivial 4-dimensional example (Example \ref{ex:poisson-and-tran-poisson}).
	\begin{defi} 
		(\cite{Filippov}) A \textbf{3-Lie algebra} is a vector space $A$ together with a skew-symmetric linear map  $[\cdot,\cdot,\cdot]:\otimes^3A\rightarrow A$ such that for all $x_i\in A, 1\leq i \leq 5$, the following Filippov-Jacobi identity holds:
		\begin{equation} \label{Jac}
			[x_1,x_2,[x_3,x_4,x_5]]=[[x_1,x_2,x_3],x_4,x_5]+[x_3,[x_1,x_2,x_4],x_5]+[x_3,x_4,[x_1,x_2,x_5]].
		\end{equation}
	\end{defi}
	A \textbf{derivation} on a 3-Lie algebra $(A,[\cdot,\cdot,\cdot])$ is a linear map $D:A\rightarrow A$ satisfying
	\[D[x_1,x_2,x_3]=[D(x_1),x_2,x_3]+[x_1,D(x_2),x_3]+[x_1,x_2,D(x_3)],\quad \forall~~ x_1,x_2,x_3\in A.\]
	For all $x_1,x_2\in A$, the operator 
	\begin{equation}\label{adx1x2}
		\mathrm{ad}_{x_1,x_2}:A\rightarrow A,\quad \mathrm{ad}_{x_1,x_2}(x):=[x_1,x_2,x],\quad \forall~ x \in A,
	\end{equation}
	is a derivation of the 3-Lie algebra $(A,[\cdot,\cdot,\cdot])$. 
	\begin{defi}
		(\cite{Kasymov}) A \textbf{representation of 3-Lie algebra} $(A,[\cdot,\cdot,\cdot])$ on a vector space $V$ is a skew-symmetric linear map $\rho:\otimes^2A\rightarrow \mathfrak{gl}(V)$ such that for all $x_i\in A, 1\leq i\leq 4$,
		\begin{eqnarray*}
			&&\rho(x_1, x_2)\rho(x_3, x_4)-\rho(x_3, x_4)\rho(x_1, x_2) = \rho([x_1, x_2, x_3], x_4)-\rho([x_1, x_2, x_4], x_3),\\
			&&\rho([x_1, x_2, x_3], x_4) =\rho(x_1, x_2)\rho(x_3, x_4) + \rho(x_2, x_3)\rho(x_1, x_4) + \rho(x_3, x_1)\rho(x_2, x_4).
		\end{eqnarray*}	
	\end{defi}
	\begin{pro}\label{pro:3-lie-dual-rep}
		$($\cite{Bai2019}$)$ Let $(V, \rho)$ be a representation of 3-Lie algebra $(A,[\cdot,\cdot,\cdot])$. Define $\rho^*:\otimes^2A\rightarrow\mathfrak{gl}(V^*)$ by
		\begin{equation}\label{3-lie-dual}
			\langle \rho^{*}(x_1,x_2)\xi,v\rangle=-\langle \xi,\rho(x_1,x_2)v\rangle,\quad\forall~ x_1,x_2\in A, \xi\in V^{*}, v\in V. 
		\end{equation}
		Then $(V^{*}, \rho^{*})$ is a representation of $(A,[\cdot,\cdot,\cdot])$, called the dual representation of $(V,\rho)$.
	\end{pro}
	\begin{ex}
		(\cite{Kasymov}) Let $(A,[\cdot,\cdot,\cdot])$ be a 3-Lie algebra. The linear map $\mathrm{ad}:\otimes^2A\rightarrow \mathfrak{gl}(A)$ is given by \eqref{adx1x2} defines a representation $(A, \mathrm{ad})$ which is called the \textbf{adjoint representation of 3-Lie algebras}. The dual representation $(A^*, \mathrm{ad}^*)$ of the adjoint representation $(A, \mathrm{ad})$ is called the \textbf{coadjoint representation of 3-Lie algebras}.
	\end{ex}
	\begin{defi}
		(\cite{Bai2010}) A \textbf{representation of a commutative associative algebra}  $(A,\cdot)$ on a vector space $V$ is a linear map $\mu:A\rightarrow \mathfrak{gl}(V)$ satisfying 
		\[\mu(x\cdot y)=\mu(x)\mu(y),\quad \forall~ x,y \in A.\]
		\begin{pro}\label{pro:com-asso-dual-rep}
			$($\cite{Bai2010}$)$ Let $(V,\mu)$ be a representation of a commutative associative algebra $(A,\cdot)$. Define $\mu^*:A\rightarrow\mathfrak{gl}(V^*)$ by
			\begin{equation}\label{com-asso-dual}
				\langle \mu^{*}(x)\xi,v\rangle=-\langle \xi,\mu(x)v\rangle,\quad\forall~ x\in A, \xi\in V^{*}, v\in V.
			\end{equation}
			Then $(V^{*}, -\mu^{*})$ is a representation of $(A,\cdot)$, called the dual representation of $(V,\mu)$.
		\end{pro}		
	\end{defi}
	\begin{ex}
		Let $(A,\cdot)$ be a commutative associative algebra. For any $x,y\in A$, the linear map $\mathcal{L}:A\rightarrow\mathfrak{gl}(A)$ defined by $x\mapsto\mathcal{L}(x)$,   $\mathcal{L}(x):y\mapsto x\cdot y$,  is a representation of $(A,\cdot)$ which is called the \textbf{adjoint representation of commutative associative algebras}. The dual representation $(A^*, -\mathcal{L}^*)$ of the adjoint representation $(A, \mathcal{L})$ is called the \textbf{coadjoint representation of commutative associative algebras}.
	\end{ex}
	\begin{defi}
		(\cite{A. Makhlouf}) A \textbf{Poisson 3-Lie algebra} is a triple $(A,\cdot, [\cdot,\cdot,\cdot])$, where $(A,\cdot)$ is a commutative associative algebra and $(A, [\cdot,\cdot,\cdot])$ is a 3-Lie algebra that satisfies the following  generalized Leibniz rule:
		\begin{equation}\label{eq:poisson-alg}
			[w\cdot x,y,z]=w\cdot[x,y,z]+x\cdot[w,y,z], \quad \forall~ x, y, z, w \in A.
		\end{equation}
	\end{defi}
	\begin{defi}
		Let $(A_1,\cdot_1,[\cdot,\cdot,\cdot]_1)$ and $(A_2,\cdot_2,[\cdot,\cdot,\cdot]_2)$ be two Poisson 3-Lie algebras.
		A \textbf{homomorphism} of Poisson 3-Lie algebra is a linear map $\varphi:A_1\rightarrow A_2$ such that
		\begin{eqnarray*}
			\varphi([x, y, z]_1) = [\varphi(x),\varphi(y),\varphi(z)]_2,\quad
			\varphi(x \cdot_1 y) = \varphi(x) \cdot_2 \varphi(y), \quad\forall~ x, y, z \in A_1.
		\end{eqnarray*} 
		An isomorphism of Poisson 3-Lie algebras is an invertible homomorphism.
	\end{defi}
	\begin{ex}
		(\cite{A. Makhlouf}) Let $A$ be a vector space generated by the polynomials in the variables $x_1, x_2, x_3$. The multiplication $\cdot$ is defined in the usual way, and the ternary operation $[\cdot,\cdot,\cdot]$ is defined as follows:
		\begin{equation*}
			[ f_1,f_2,f_3] =\left| \begin{matrix}
				{\frac{\partial f_1}{\partial x_1}}&		\frac{\partial f_1}{\partial x_2}&		\frac{\partial f_1}{\partial x_3}\\
				\frac{\partial f_2}{\partial x_1}&		\frac{\partial f_2}{\partial x_2}&		\frac{\partial f_2}{\partial x_3}\\
				\frac{\partial f_3}{\partial x_1}&		\frac{\partial f_3}{\partial x_2}&		\frac{\partial f_3}{\partial x_3}\\
			\end{matrix} \right|, \quad \forall~ f_1,f_2,f_3\in A.	
		\end{equation*}
		Then $(A,\cdot, [\cdot,\cdot,\cdot])$ is a Poisson 3-Lie algebra.
	\end{ex}
	\begin{defi}
		(\cite{Bai2020}) A \textbf{transposed Poisson 3-Lie algebra} is a triple $(B,\cdot, [\cdot,\cdot,\cdot])$, where $(B,\cdot)$ is a commutative associative algebra and $(B, [\cdot,\cdot,\cdot])$ is a 3-Lie algebra that satisfies
		\begin{equation}\label{eq:trans-poisson-alg}
			3w\cdot[x,y,z]=[w\cdot x,y,z]+[x,w\cdot y,z]+[x,y,w\cdot z], \quad \forall~ x, y, z, w \in B.
		\end{equation}
	\end{defi}
	
	\begin{ex}\label{ex:trans-poisson}
		(\cite{Jiang}) Let $B$ be a 3-dimensional vector space with a basis  $\{e_1,e_2,e_3\}$. Define the nonzero operations $\cdot$ and $[\cdot,\cdot,\cdot]$ by
		\begin{equation*}
			e_2\cdot e_2=e_1,~ e_3\cdot e_3=-3e_1,~
			[e_1,e_2,e_3] =e_1.	
		\end{equation*}
		Then $(B,\cdot, [\cdot,\cdot,\cdot])$ is a transposed Poisson 3-Lie algebra.
	\end{ex}
	\begin{rmk}
		In \cite{B. Sartayev}, the author proves that every transposed Poisson algebra is an $F$-manifold algebra. However, for ternary $F$-manifold algebra as defined in \cite{Benhassine}, a transposed Poisson 3-Lie algebra is not necessarily a ternary $F$-manifold algebra.
	\end{rmk}
	\begin{pro}
		Let $(A,\cdot,[\cdot,\cdot,\cdot])$ be a transposed Poisson 3-Lie algebra. For any $h\in A$, define a trilinear map $[\cdot,\cdot,\cdot]_h$ on $A$ by
		\begin{equation}\label{mul-h}
			[x,y,z]_h:=h\cdot[x,y,z],\quad \forall~ x,y,z\in A.
		\end{equation}
		Then $(A,\cdot,[\cdot,\cdot,\cdot]_h)$ is a transposed Poisson 3-Lie algebra.
	\end{pro}
	\begin{proof}
		Since $[\cdot,\cdot,\cdot]$ is skew-symmetric, the trilinear map $[\cdot,\cdot,\cdot]_h$ is also skew-symmetric. By \cite[Theorem 2.3]{Huang}, for all $y_1,y_2,y_3,x_1,x_2\in A$, we have
		\begin{equation}
			[h\cdot[y_1,y_2,y_3],x_1,x_2]=[h\cdot[y_1,x_1,x_2],y_2,y_3]+[y_1,h\cdot[y_2,x_1,x_2],y_3]+[y_1,y_2,h\cdot[y_3,x_1,x_2]].
		\end{equation}
		Then 
		\begin{eqnarray*}
			&&[[y_1,y_2,y_3]_h,x_1,x_2]_h-[[y_1,x_1,x_2]_h,y_2,y_3]_h+[y_1,[y_2,x_1,x_2]_h,y_3]_h+[y_1,y_2,[y_3,x_1,x_2]_h]_h \\
			&\overset{\eqref{mul-h}}{=}&h\cdot([h\cdot[y_1,y_2,y_3],x_1,x_2]-[h\cdot[y_1,x_1,x_2],y_2,y_3]+[y_1,h\cdot[y_2,x_1,x_2],y_3]\\
			&&+[y_1,y_2,h\cdot[y_3,x_1,x_2]])\\
			&=&0.
		\end{eqnarray*} 
		Therefore, $(A,[\cdot,\cdot,\cdot]_h)$ is a 3-Lie algebra. Furthermore, we have
		\begin{eqnarray*}
			&&[x_1\cdot y_1,y_2,y_3]_h+[y_1,x_1\cdot y_2,y_3]_h+[y_1,y_2,x_1\cdot y_3]_h\\
			&\overset{\eqref{mul-h}}{=}&h\cdot([x_1\cdot y_1,y_2,y_3]+[y_1,x_1\cdot y_2,y_3]+[y_1,y_2,x_1\cdot y_3])\\
			&\overset{\eqref{eq:trans-poisson-alg}}{=}&3h\cdot x_1\cdot[y_1,y_2,y_3]\\
			&\overset{\eqref{mul-h}}{=}&3x_1\cdot[y_1,y_2,y_3]_h, \quad \forall~ ~ y_1, y_2, y_3, x_1\in A.
		\end{eqnarray*}
		Hence $(A,\cdot,[\cdot,\cdot,\cdot]_h)$ is a transposed Poisson 3-Lie algebra.
	\end{proof}
	\begin{pro} \label{pro:both-trans-poisson-and-poisson}
		$($\cite{Bai2020}$)$	Let $(A,\cdot)$ be a commutative associative algebra and $(A,[\cdot,\cdot,\cdot])$ be a 3-Lie algebra. Then $(A,\cdot,[\cdot,\cdot,\cdot])$ is both a Poisson 3-Lie algebra and a transposed Poisson 3-Lie algebra if and only if 
		\begin{equation}\label{eq:trans-poisson-and-poisson}
			u\cdot[x,y,z]=[u\cdot x,y,
			z]=0, \quad \forall~ x,y,z,u\in A.
		\end{equation}
	\end{pro}	
	\begin{defi} \label{ad-Tran}
		Let $(A,\cdot)$ be a commutative associative algebra and $(A,[\cdot,\cdot,\cdot])$ be a 3-Lie algebra. If $\cdot$ and $[\cdot,\cdot,\cdot]$ satisfy \eqref{eq:trans-poisson-and-poisson}, then the triple $(A,\cdot,[\cdot,\cdot,\cdot])$ is called an \textbf{ admissible transposed Poisson 3-Lie algebra}.
	\end{defi}
	Example \ref{ex:trans-poisson} is not a Poisson 3-Lie algebra. Next we provide a non-trivial example  that is both a Poisson 3-Lie algebra and a transposed Poisson 3-Lie algebra.
	\begin{ex}\label{ex:poisson-and-tran-poisson}
		Let $A$ be a 4-dimensional vector space with a basis $\{e_1,e_2,e_3,e_4\}$. Define the  nonzero operations $\cdot$ and $[\cdot,\cdot,\cdot]$ on $A$ by
		$$[e_2,e_3,e_4]=e_1, \quad e_2\cdot e_3=e_3\cdot e_2=e_1.$$
		Then $(A,\cdot,[\cdot,\cdot,\cdot])$ is both a Poisson 3-Lie algebra and a transposed Poisson 3-Lie algebra. 
	\end{ex}	
	\section{Direct sum and tensor product of (transposed) Poisson 3-Lie algebras}\label{sec:properties}
	In this section, we construct Poisson 3-Lie algebras and transposed Poisson 3-Lie algebras on the direct sum and tensor product of vector spaces.
	\subsection{Direct sum and tensor product of  Poisson 3-Lie algebras}
	\begin{pro}\label{pro:poisson-direct-sum}
		Let $(A_1,\cdot_1, [\cdot,\cdot,\cdot]_1)$ and $(A_2,\cdot_2, [\cdot,\cdot,\cdot]_2)$ be two Poisson 3-Lie algebras. For all $x_i\in A_1, y_i\in A_2, 1\leq i\leq3$, define linear maps $\cdot_{A_1 \oplus A_2}:\otimes^2(A_1 \oplus A_2)\rightarrow A_1 \oplus A_2$ and $[\cdot,\cdot,\cdot]_{A_1 \oplus A_2}:\otimes^3(A_1 \oplus A_2)\rightarrow A_1 \oplus A_2$ by
		\begin{align}
			(x_1+y_1)\cdot_{A_1 \oplus A_2}(x_2+y_2)&=x_1\cdot_1 x_2+y_1\cdot_2 y_2,\label{eq:poisson-direct-sum1} \\ 
			[x_1+y_1,x_2+y_2,x_3+y_3]_{A_1 \oplus A_2}&=[x_1,x_2,x_3]_1 + [y_1,y_2,y_3]_2.\label{eq:poisson-direct-sum2}
		\end{align}
		Then $(A_1 \oplus A_2,\cdot_{A_1 \oplus A_2}, [\cdot,\cdot,\cdot]_{A_1 \oplus A_2})$ is a Poisson 3-Lie algebra.
	\end{pro}
	\begin{proof}
		For all $x_i\in A_1, y_i\in A_2, 1\leq i\leq5$,  we have 
		\begin{eqnarray*}
			(x_1+y_1)\cdot_{A_1 \oplus A_2}(x_2+y_2)&\overset{\eqref{eq:poisson-direct-sum1}}{=}&x_1\cdot_1 x_2+y_1\cdot_2 y_2
			\\&=&x_2\cdot_1 x_1+y_2\cdot_2 y_1\\
			&\overset{\eqref{eq:poisson-direct-sum1}}{=}&(x_2+y_2)\cdot_{A_1 \oplus A_2}(x_1+y_1),
		\end{eqnarray*} 
		and
		\begin{eqnarray*}
			&&\big((x_1+y_1)\cdot_{A_1 \oplus A_2}(x_2+y_2)\big)\cdot_{A_1 \oplus A_2}(x_3+y_3)
			\\&\overset{\eqref{eq:poisson-direct-sum1}}{=}&(x_1\cdot_1 x_2+y_1\cdot_2 y_2)\cdot_{A_1 \oplus A_2}(x_3+y_3)
			\\&\overset{\eqref{eq:poisson-direct-sum1}}{=}&x_1\cdot_1 x_2\cdot_1 x_3+y_1\cdot_2 y_2\cdot_2 y_3
			\\&\overset{\eqref{eq:poisson-direct-sum1}}{=}&(x_1+y_1)\cdot_{A_1 \oplus A_2}(x_2\cdot_1 x_3+y_2\cdot_2 y_3)
			\\&\overset{\eqref{eq:poisson-direct-sum1}}{=}&(x_1+y_1)\cdot_{A_1 \oplus A_2}\big((x_2+y_2)\cdot_{A_1 \oplus A_2}(x_3+y_3)\big).
		\end{eqnarray*} 
		Therefore $(A_1 \oplus A_2,\cdot_{A_1 \oplus A_2})$ is a commutative associative algebra.  Since the linear maps $[\cdot,\cdot,\cdot]_1$ and $[\cdot,\cdot,\cdot]_2$ are skew-symmetric, the linear map $[\cdot,\cdot,\cdot]_{A_1 \oplus A_2}$ is also skew-symmetric. By \eqref{Jac} and \eqref{eq:poisson-direct-sum2} we have
		\begin{eqnarray*}
			&&[[x_1+y_1,x_2+y_2,x_3+y_3]_{A_1 \oplus A_2},x_4+y_4,x_5+y_5]_{A_1 \oplus A_2}
			\\&\overset{\eqref{eq:poisson-direct-sum2}}{=}&[[x_1,x_2,x_3]_1 + [y_1,y_2,y_3]_2,x_4+y_4,x_5+y_5]_{A_1 \oplus A_2}
			\\&\overset{\eqref{eq:poisson-direct-sum2}}{=}&[[x_1,x_2,x_3]_1 ,x_4,x_5]_1+[[y_1,y_2,y_3]_2,y_4,y_5]_2
			\\&\overset{\eqref{Jac}}{=}&[[x_1,x_4,x_5]_1 ,x_2,x_3]_1+[x_1,[x_2,x_4,x_5]_1 ,x_3]_1+[x_1,x_2,[x_3,x_4,x_5]_1 ]_1
			\\&&+[[y_1,y_4,y_5]_2,y_2,y_3]_2+[y_1,[y_2,y_4,y_5]_2,y_3]_2+[y_1,y_2,[y_3,y_4,y_5]_2]_2
			\\&\overset{\eqref{eq:poisson-direct-sum2}}{=}&[[x_1,x_4,x_5]_1+[y_1,y_4,y_5]_2,x_2+y_2,x_3+y_3]_{A_1 \oplus A_2}
			\\&&+[x_1+y_1,[x_2,x_4,x_5]_1+[y_2,y_4,y_5]_2,x_3+y_3]_{A_1 \oplus A_2}
			\\&&+[x_1+y_1,x_2+y_2,[x_3,x_4,x_5]_1+[y_3,y_4,y_5]_2]_{A_1 \oplus A_2}
			\\&\overset{\eqref{eq:poisson-direct-sum2}}{=}&[[x_1+y_1,x_4+y_4,x_5+y_5]_{A_1 \oplus A_2},x_2+y_2,x_3+y_3]_{A_1 \oplus A_2}
			\\&&+[x_1+y_1,[x_2+y_2,x_4+y_4,x_5+y_5]_{A_1 \oplus A_2},x_3+y_3]_{A_1 \oplus A_2}
			\\&&+[x_1+y_1,x_2+y_2,[x_3+y_3,x_4+y_4,x_5+y_5]_{A_1 \oplus A_2}]_{A_1 \oplus A_2}.
		\end{eqnarray*} 	
		Hence $(A_1 \oplus A_2,[\cdot,\cdot,\cdot]_{A_1 \oplus A_2})$ is a 3-Lie algebra. Furthermore, we have
		\begin{eqnarray*}
			&&[(x_4+y_4)\cdot_{A_1 \oplus A_2}(x_1+y_1),x_2+y_2,x_3+y_3]_{A_1 \oplus A_2}\\
			&\overset{\eqref{eq:poisson-direct-sum1}}{=}&[x_4\cdot_1 x_1+y_4\cdot_2 y_1,x_2+y_2,x_3+y_3]_{A_1 \oplus A_2}\\
			&=&[x_4\cdot_1 x_1,x_2,x_3]_1+[y_4\cdot_2 y_1,y_2,y_3]_2\\
			&\overset{\eqref{eq:poisson-alg}}{=}&x_4\cdot_1 [x_1,x_2,x_3]_1+x_1\cdot_1 [x_4,x_2,x_3]_1+y_4\cdot_2 [y_1,y_2,y_3]_2+y_1\cdot_2 [y_4,y_2,y_3]_2\\
			&\overset{\eqref{eq:poisson-direct-sum1}}{=}&(x_1+y_1)\cdot_{A_1 \oplus A_2}([x_4,x_2,x_3]_1+[y_4,y_2,y_3]_2)\\
			&&+(x_4+y_4)\cdot_{A_1 \oplus A_2}([x_1,x_2,x_3]_1+[y_1,y_2,y_3]_2)\\
			&\overset{\eqref{eq:poisson-direct-sum2}}{=}&(x_1+y_1)\cdot_{A_1 \oplus A_2}[x_4+y_4,x_2+y_2,x_3+y_3]_{A_1 \oplus A_2}\\
			&&+(x_4+y_4)\cdot_{A_1 \oplus A_2}[x_1+y_1,x_2+y_2,x_3+y_3]_{A_1 \oplus A_2}.
		\end{eqnarray*}
		Therefore, $(A_1 \oplus A_2,\cdot_{A_1 \oplus A_2}, [\cdot,\cdot,\cdot]_{A_1 \oplus A_2})$ is a Poisson 3-Lie algebra.
	\end{proof}	
	In \cite{Bai2020}, the tensor product of two Poisson algebras $(A_1,\cdot_1,[\cdot,\cdot]_1)$ and $(A_2,\cdot_2,[\cdot,\cdot]_2)$ is a Poisson algebra $(A_1\otimes A_2,\cdot_{A_1\otimes A_2},[\cdot,\cdot]_{A_1\otimes A_2})$, where $\cdot_{A_1\otimes A_2}$ and $[\cdot,\cdot]_{A_1\oplus A_2}$ are defined by
	\begin{align}
		[x_1\otimes y_1,x_2\otimes y_2]_{A_1\otimes A_2}&=[x_1,x_2]_1 \otimes (y_1\cdot_2 y_2)+(x_1\cdot_1 x_2) \otimes [y_1, y_2]_2,\label{eq:direct_sum}\\
		(x_1\otimes y_1) \cdot_{A_1\oplus A_2} (x_2\otimes y_2) &= (x_1 \cdot_1 x_2) \otimes (y_1 \cdot_2 y_2), \quad \forall~ x_1,x_2\in A_1, y_1,y_2 \in A_2.
	\end{align}	
	If replace the Lie bracket in \eqref{eq:direct_sum} with 3-Lie bracket, then the tensor product of two Poisson 3-Lie algebras is not a 3-Lie algebra. However, the tensor product of a ternary $F$-manifold algebra and a commutative associative algebra is also a ternary $F$-manifold algebra \cite{Benhassine}. We  similarly  construct the tensor product of a Poisson 3-Lie algebra and a commutative associative algebra as follows.
	\begin{pro}\label{pro:poisson-tensor}
		Let $(A_1,\cdot_1,[\cdot,\cdot,\cdot]_1)$ be a Poisson 3-Lie algebra and $(A_2,\cdot_2)$ be a commutative associative algebra. For all $x_i\in A_1, y_i\in A_2, 1\leq i\leq3$, define linear maps $\cdot_{A_1\otimes A_2}:\otimes^2(A_1\otimes A_2)\rightarrow A_1\otimes A_2$ and $[\cdot,\cdot,\cdot]_{A_1\otimes A_2}:\otimes^3(A_1\otimes A_2)\rightarrow A_1\otimes A_2$ by
		\begin{align}
			(x_1\otimes y_1)\cdot_{A_1\otimes A_2}(x_2\otimes y_2)&=(x_1\cdot_1 x_2)\otimes (y_1\cdot_2 y_2),\label{eq:poisson-tensor-1}\\
			[x_1\otimes y_1,x_2\otimes y_2,x_3\otimes y_3]_{A_1\otimes A_2}&=[x_1,x_2,x_3]_1 \otimes (y_1\cdot_2 y_2\cdot_2 y_3).\label{eq:poisson-tensor-2}
		\end{align}
		Then $(A_1\otimes A_2, \cdot_{A_1\otimes A_2}, [\cdot,\cdot,\cdot]_{A_1\otimes A_2})$ is a Poisson 3-Lie algebra.
	\end{pro}
	\begin{proof}
		For all $x_i\in A_1, y_i\in A_2, 1\leq i\leq5$, we have
		\begin{eqnarray*}
			(x_1\otimes y_1)\cdot_{A_1\otimes A_2}(x_2\otimes y_2)&\overset{\eqref{eq:poisson-tensor-1}}{=}&(x_1\cdot_1 x_2)\otimes( y_1\cdot_2 y_2)\\
			&=&(x_2\cdot_1 x_1)\otimes (y_2\cdot_2 y_1)\\
			&\overset{\eqref{eq:poisson-tensor-1}}{=}&(x_2\otimes y_2)\cdot_{A_1\otimes A_2}(x_1\otimes y_1),
		\end{eqnarray*} 
		and
		\begin{eqnarray*}
			&&\big((x_1\otimes y_1)\cdot_{A_1\otimes A_2}(x_2\otimes y_2)\big)\cdot_{A_1\otimes A_2}(x_3\otimes y_3)\\
			&\overset{\eqref{eq:poisson-tensor-1}}{=}&\big((x_1\cdot_1 x_2)\otimes (y_1\cdot_2 y_2)\big)\cdot_{A_1\otimes A_2}(x_3\otimes y_3)
			\\&\overset{\eqref{eq:poisson-tensor-1}}{=}&(x_1\cdot_1 x_2\cdot_1 x_3)\otimes (y_1\cdot_2 y_2\cdot_2 y_3)
			\\&\overset{\eqref{eq:poisson-tensor-1}}{=}&(x_1\otimes y_1)\cdot_{A_1\otimes A_2}\big((x_2\cdot_1 x_3)\otimes (y_2\cdot_2 y_3)\big)
			\\&\overset{\eqref{eq:poisson-tensor-1}}{=}&(x_1\otimes y_1)\cdot_{A_1\otimes A_2}\big((x_2\otimes y_2)\cdot_{A_1\otimes A_2}(x_3\otimes y_3)\big).
		\end{eqnarray*}	
		Therefore $(A_1\otimes A_2,\cdot_{A_1\otimes A_2})$ is a commutative associative algebra. Since the linear map $[\cdot,\cdot,\cdot]_1$ is skew-symmetric, the linear map $[\cdot,\cdot,\cdot]_{A_1\otimes A_2}$ is also skew-symmetric. Moreover, we have
		\begin{eqnarray*}
			&&[[x_1\otimes y_1,x_2\otimes y_2,x_3\otimes y_3]_{A_1\otimes A_2},x_4\otimes y_4,x_5\otimes y_5]_{A_1\otimes A_2}
			\\&\overset{\eqref{eq:poisson-tensor-2}}{=}&[[x_1,x_2,x_3]_1 \otimes (y_1\cdot_2 y_2\cdot_2 y_3),x_4\otimes y_4,x_5\otimes y_5]_{A_1\otimes A_2}
			\\&\overset{\eqref{eq:poisson-tensor-2}}{=}&[[x_1,x_2,x_3]_1 ,x_4,x_5]_1\otimes (y_1\cdot_2 y_2\cdot_2 y_3\cdot_2 y_4\cdot_2 y_5)
			\\&\overset{\eqref{Jac}}{=}&[[x_1,x_4,x_5]_1 ,x_2,x_3]_1\otimes(y_1\cdot_2 y_2\cdot_2 y_3\cdot_2 y_4\cdot_2 y_5)
			\\&&+[x_1,[x_2,x_4,x_5]_1 ,x_3]_1\otimes(y_1\cdot_2 y_2\cdot_2 y_3\cdot_2 y_4\cdot_2 y_5)
			\\&&+[x_1,x_2,[x_3,x_4,x_5]_1 ]_1\otimes(y_1\cdot_2 y_2\cdot_2 y_3\cdot_2 y_4\cdot_2 y_5)
			\\&\overset{\eqref{eq:poisson-tensor-2}}{=}&[[x_1,x_4,x_5]_1\otimes y_1\cdot_2 y_4\cdot_2 y_5,x_2\otimes y_2,x_3\otimes y_3]_{A_1\otimes A_2}
			\\&&+[x_1\otimes y_1,[x_2,x_4,x_5]_1\otimes y_2\cdot_2 y_4\cdot_2 y_5,x_3\otimes y_3]_{A_1\otimes A_2}
			\\&&+[x_1\otimes y_1,x_2\otimes y_2,[x_3,x_4,x_5]_1\otimes y_3\cdot_2 y_4\cdot_2 y_5]_{A_1\otimes A_2}
			\\&\overset{\eqref{eq:poisson-tensor-2}}{=}&[[x_1\otimes y_1,x_4\otimes y_4,x_5\otimes y_5]_{A_1\otimes A_2},x_2\otimes y_2,x_3\otimes y_3]_{A_1\otimes A_2}
			\\&&+[x_1\otimes y_1,[x_2\otimes y_2,x_4\otimes y_4,x_5\otimes y_5]_{A_1\otimes A_2},x_3\otimes y_3]_{A_1\otimes A_2}
			\\&&+[x_1\otimes y_1,x_2\otimes y_2,[x_3\otimes y_3,x_4\otimes y_4,x_5\otimes y_5]_{A_1\otimes A_2}]_{A_1\otimes A_2}.
		\end{eqnarray*}
		Hence  $(A_1\otimes A_2,[\cdot,\cdot,\cdot]_{A_1\otimes A_2})$ is a 3-Lie algebra. Furthermore, we have
		\begin{eqnarray*}
			&&[(x_4\otimes y_4)\cdot_{A_1\otimes A_2}(x_1\otimes y_1),x_2\otimes y_2,x_3\otimes y_3]_{A_1\otimes A_2}\\
			&\overset{\eqref{eq:poisson-tensor-1}}{=}&[(x_4\cdot_1 x_1)\otimes( y_4\cdot_2 y_1),x_2\otimes y_2,x_3\otimes y_3]_{A_1\otimes A_2}\\
			&\overset{\eqref{eq:poisson-tensor-2}}{=}&[x_4\cdot_1 x_1,x_2,x_3]_1\otimes \big((y_4\cdot_2 y_1)\cdot_2 y_2\cdot_2 y_3\big)\\
			&\overset{\eqref{eq:poisson-alg}}{=}&\left(x_4\cdot_1 [x_1,x_2,x_3]_1+x_1\cdot_1[x_4,x_2,x_3]_1\right)\otimes(y_1\cdot_2 y_2\cdot_2 y_3\cdot_2 y_4)\\
			&\overset{\eqref{eq:poisson-tensor-1}}{=}&(x_4\cdot_1[x_1,x_2,x_3]_1)\otimes \big(y_4\cdot_2(y_1\cdot_2 y_2\cdot_2 y_3)\big)+x_1\cdot_1[x_4,x_2,x_3]_1\otimes \big(y_1\cdot_2(y_4\cdot_2 y_2\cdot_2 y_3)\big)\\	
			&\overset{\eqref{eq:poisson-tensor-2}}{=}&(x_4\otimes y_4)\cdot_{A_1\otimes A_2}[x_1\otimes y_1,x_2\otimes y_2,x_3\otimes y_3]_{A_1\otimes A_2}\\
			&&+(x_1\otimes y_1)\cdot_{A_1\otimes A_2}[x_4\otimes y_4,x_2\otimes y_2,x_3\otimes y_3]_{A_1\otimes A_2}.
		\end{eqnarray*}
		Therefore,	$(A_1\otimes A_2,\cdot_{A_1\otimes A_2},[\cdot,\cdot,\cdot]_{A_1\otimes A_2})$ is a Poisson 3-Lie algebra. 
	\end{proof}
	\subsection{Direct sum and tensor product of transposed Poisson 3-Lie algebras}
	\begin{pro}
		Let $(B_1,\cdot_1, [\cdot,\cdot,\cdot]_1)$ and $(B_2,\cdot_2, [\cdot,\cdot,\cdot]_2)$ be two transposed Poisson 3-Lie algebras. For all $x_i\in B_1, y_i\in B_2, 1\leq i\leq3$, define linear maps $\cdot_{B_1 \oplus B_2}:\otimes^2(B_1 \oplus B_2)\rightarrow B_1 \oplus B_2$ and $[\cdot,\cdot,\cdot]_{B_1 \oplus B_2}:\otimes^3(B_1 \oplus B_2)\rightarrow B_1 \oplus B_2$ by
		\begin{align}
			(x_1+y_1)\cdot_{B_1 \oplus B_2}(x_2+y_2)&=x_1\cdot_1 x_2+y_1\cdot_2 y_2,\label{direct-sum1} \\ 
			[x_1+y_1,x_2+y_2,x_3+y_3]_{B_1 \oplus B_2}&=[x_1,x_2,x_3]_1 + [y_1,y_2,y_3]_2.\label{direct-sum2}
		\end{align}
		Then $(B_1 \oplus B_2,\cdot_{B_1 \oplus B_2}, [\cdot,\cdot,\cdot]_{B_1 \oplus B_2})$ is a transposed Poisson 3-Lie algebra.
	\end{pro}
	\begin{proof}
		By the proof of Proposition \ref{pro:poisson-direct-sum}, we have that $(B_1 \oplus B_2,\cdot_{B_1 \oplus B_2})$ is a commutative associative algebra and $(B_1 \oplus B_2,[\cdot,\cdot,\cdot]_{B_1 \oplus B_2})$ is a 3-Lie algebra.	Now we  prove  that \eqref{eq:trans-poisson-alg} holds. For all $x_i\in B_1, y_i\in B_2, 1\leq i\leq4$, we have
		%	\delete{, for all $x_i\in A, y_i\in B,\forall~ i=1,2,3,4$ we have\eqref{direct-sum2}	} 	 
		\begin{eqnarray*}
			&&3(x_4+y_4)\cdot_{B_1 \oplus B_2}[x_1+y_1,x_2+y_2,x_3+y_3]_{B_1 \oplus B_2}\\
			&\overset{\eqref{direct-sum2}}{=}&3(x_4+y_4)\cdot_{B_1 \oplus B_2}([x_1,x_2,x_3]_1+[y_1,y_2,y_3]_2)\\
			&\overset{\eqref{direct-sum1}}{=}&3x_4\cdot_1[x_1,x_2,x_3]_1+3y_4\cdot_2[y_1,y_2,y_3]_2\\
			&\overset{\eqref{eq:trans-poisson-alg}}{=}&[x_4\cdot_1 x_1,x_2,x_3]_1+[x_1,x_4\cdot_1 x_2,x_3]_1+[x_1,x_2,x_4\cdot_1 x_3]_1+[y_4\cdot_2 y_1,y_2,y_3]_2\\
			&&+[y_1,y_4\cdot_2 y_2,y_3]_2+[y_1,y_2,y_4\cdot_2 y_3]_2\\
			&\overset{\eqref{direct-sum2}}{=}&[x_4\cdot_1 x_1+y_4\cdot_2 y_1,x_2+y_2,x_3+y_3]_{B_1 \oplus B_2}+[x_1+y_1,x_4\cdot_1 x_2+y_4\cdot_2 y_2,x_3+y_3]_{B_1 \oplus B_2}\\
			&&+[x_1+y_1,x_2+y_2,x_4\cdot_1 x_3+y_4\cdot_2 y_3]_{B_1 \oplus B_2}\\
			&\overset{\eqref{direct-sum1}}{=}&[(x_4+y_4)\cdot_{B_1 \oplus B_2}(x_1+y_1),x_2+y_2,x_3+y_3]_{B_1 \oplus B_2}\\
			&&+[x_1+y_1,(x_4+y_4)\cdot_{B_1 \oplus B_2}(x_2+y_2),x_3+y_3]_{B_1 \oplus B_2}\\
			&&+[x_1+y_1,x_2+y_2,(x_4+y_4)\cdot_{B_1 \oplus B_2}(x_3+y_3)]_{B_1 \oplus B_2}.
		\end{eqnarray*}
		Therefore $(B_1 \oplus B_2,\cdot_{B_1 \oplus B_2}, [\cdot,\cdot,\cdot]_{B_1 \oplus B_2})$ is a transposed Poisson 3-Lie algebra. 
	\end{proof}
	\begin{pro}
		Let $(B_1,\cdot_1,[\cdot,\cdot,\cdot]_1)$ be a transposed Poisson 3-Lie algebra and $(B_2,\cdot_2)$ be a commutative associative algebra. For all $x_i\in B_1, y_i\in B_2, 1\leq i\leq3$, define linear maps $\cdot_{B_1 \otimes B_2}:\otimes^2(B_1 \otimes B_2)\rightarrow B_1 \otimes B_2$ and $[\cdot,\cdot,\cdot]_{B_1 \otimes B_2}:\otimes^3(B_1 \otimes B_2)\rightarrow B_1 \otimes B_2$ by
		\begin{align}
			(x_1\otimes y_1)\cdot_{B_1 \otimes B_2}(x_2\otimes y_2)&=(x_1\cdot_1 x_2)\otimes (y_1\cdot_2 y_2),\label{tensor-1}\\
			[x_1\otimes y_1,x_2\otimes y_2,x_3\otimes y_3]_{B_1 \otimes B_2}&=[x_1,x_2,x_3]_1 \otimes (y_1\cdot_2 y_2\cdot_2 y_3).\label{tensor-2}
		\end{align}
		Then $(B_1 \otimes B_2,\cdot_{B_1 \otimes B_2},[\cdot,\cdot,\cdot]_{B_1 \otimes B_2})$ is a transposed Poisson 3-Lie algebra.
	\end{pro}
	\begin{proof}
		By the proof of Proposition \ref{pro:poisson-tensor}, we have $(B_1 \otimes B_2,\cdot_{B_1 \otimes B_2})$ is a commutative associative algebra and $(B_1 \otimes B_2,[\cdot,\cdot,\cdot]_{B_1 \otimes B_2})$ is a 3-Lie algebra. Next we prove that \eqref{eq:trans-poisson-alg} holds. For all $x_i\in A, y_i\in B, 1\leq i\leq4$, we have
		\begin{eqnarray*}
			&&3(x_4\otimes y_4)\cdot_{B_1 \otimes B_2}[x_1\otimes y_1,x_2\otimes y_2,x_3\otimes y_3]_{B_1 \otimes B_2}\\
			&\overset{\eqref{tensor-2}}{=}&3(x_4\otimes y_4)\cdot_{B_1 \otimes B_2}\big([x_1,x_2,x_3]_1\otimes(y_1\cdot_2 y_2\cdot_2 y_3)\big)\\
			&\overset{\eqref{tensor-1}}{=}&3(x_4\cdot_{A}[x_1,x_2,x_3]_1)\otimes \big(y_4\cdot_{B}(y_1\cdot_2 y_2\cdot_2 y_3)\big)\\
			&\overset{\eqref{eq:trans-poisson-alg}}{=}&\left([x_4\cdot_1 x_1,x_2,x_3]_1+[x_1,x_4\cdot_1 x_2,x_3]_1+[x_1,x_2,x_4\cdot_1 x_3]_1\right)\otimes(y_1\cdot_2 y_2\cdot_2 y_3\cdot_2 y_4)\\
			&\overset{\eqref{tensor-2}}{=}&[(x_4\cdot_1 x_1)\otimes (y_4\cdot_2 y_1),x_2\otimes y_2,x_3\otimes y_3]_{B_1 \otimes B_2}\\
			&&+[x_1\otimes y_1,(x_4\cdot_1 x_2)\otimes (y_4\cdot_2 y_2),x_3\otimes y_3]_{B_1 \otimes B_2}\\
			&&+[x_1\otimes y_1,x_2\otimes y_2,(x_4\cdot_1 x_3)\otimes (y_4\cdot_2 y_3)]_{B_1 \otimes B_2}\\
			&\overset{\eqref{tensor-1}}{=}&[(x_4\otimes y_4)\cdot_{B_1 \otimes B_2}(x_1\otimes y_1),x_2\otimes y_2,x_3\otimes y_3]_{B_1 \otimes B_2}\\
			&&+[x_1\otimes y_1,(x_4\otimes y_4)\cdot_{B_1 \otimes B_2}(x_2\otimes y_2),x_3\otimes y_3]_{B_1 \otimes B_2}\\
			&&+[x_1\otimes y_1,x_2\otimes y_2,(x_4\otimes y_4)\cdot_{B_1 \otimes B_2}(x_3\otimes y_3)]_{B_1 \otimes B_2}.
		\end{eqnarray*}
		Therefore, $(B_1 \otimes B_2,\cdot_{B_1 \otimes B_2},[\cdot,\cdot,\cdot]_{B_1 \otimes B_2})$ is a transposed Poisson 3-Lie algebra. 
	\end{proof}
	\section{Double construction Poisson 3-Lie bialgebras}\label{sec:double-Poisson}
	In this section, we first introduce the representations and dual representations of Poisson 3-Lie algebras, construct the semi-direct product Poisson 3-Lie algebras (Proposition \ref{pro:possion-semi-direct}). 
	Then we define matched pairs and Manin triples of Poisson 3-Lie algebras.  We also discuss the double construction Poisson 3-Lie bialgebras and establish equivalences between these structures (Theorem \ref{thm:one-one-co}). 
	\subsection{Representations of Poisson 3-Lie algebras}\label{subsec:41}
	Recall \cite{X. Ni}, a representation of Poisson algebra $(P, \cdot, [\cdot,\cdot])$ is a triple $(V, \rho, \mu)$, where $(V, \rho)$ and $(V, \mu)$ are respectively representations of Lie algebra $(P, [\cdot,\cdot])$ and commutative associative algebra $(P,\cdot)$ that satisfy
	\begin{align*}
		\rho(x\cdot y)&=\mu(y)\rho(x)+\mu(x)\rho(y),\\
		\mu([x,y])&=\rho(x)\mu(y)-\mu(y)\rho(x), \quad\forall~ x, y\in P.
	\end{align*}
	Replace the above representation of Lie algebra with the representation of 3-Lie algebra. By \eqref{eq:poisson-alg}, we give the representation of the Poisson 3-Lie algebra.
	\begin{defi}
		A \textbf{representation of Poisson 3-Lie algebra} $(A,\cdot, [\cdot,\cdot,\cdot])$ is a triple $(V,\rho,\mu)$, where $(V,\mu)$ is a representation of the commutative associative algebra $(A,\cdot)$ and $(V,\rho)$ is a representation of the 3-Lie algebra $(A, [\cdot,\cdot,\cdot])$ such that
		\begin{align}
			\rho(x\cdot y,z)&=\mu(x)\rho(y,z)+\mu(y)\rho(x,z), \label{rep-1} \\
			\rho(x,y)\mu(z)&=\mu([x,y,z])+\mu(z)\rho(x,y), \quad \forall~ x,y,z\in A. \label{rep-2}
		\end{align}
	\end{defi}	
	\begin{rmk}
		Ncib provides the representation of $n$-ary Hom-Nambu Poisson superalgebras \cite{O. Ncib}. The representations of a Poisson 3-Lie algebras are a special case when $n=3$.
	\end{rmk}
	Let $(A,\cdot, [\cdot,\cdot,\cdot])$ be a Poisson 3-Lie algebra, $(A,\mathrm{ad})$ and $(A,\mathcal{L})$ be the adjoint representations of $(A, [\cdot,\cdot,\cdot])$ and $(A,\cdot)$, respectively. Then $(A,\mathrm{ad},\mathcal{L})$ is a representation of $(A,\cdot, [\cdot,\cdot,\cdot])$, which is called the \textbf{adjoint representation of Poisson 3-Lie algebras}.

	\begin{pro}\label{pro:poisson-dual}
		Let $(A,\cdot, [\cdot,\cdot,\cdot])$ be a Poisson 3-Lie algebra and $(V,\rho,\mu)$ be a representation of $(A,\cdot, [\cdot,\cdot,\cdot])$. Then $(V^*,\rho^*,-\mu^*)$ is a representation of $(A,\cdot, [\cdot,\cdot,\cdot])$. This representation is called the \textbf{dual representation} of $(V,\rho,\mu)$. In particular,  $(A^*,\mathrm{ad}^*,-\mathcal{L}^*)$ is a representation of $(A,\cdot, [\cdot,\cdot,\cdot])$, which is called the \textbf{coadjoint representation of Poisson 3-Lie algebras}. 
	\end{pro}
	\begin{proof}
		By Proposition \ref{pro:3-lie-dual-rep} and Proposition \ref{pro:com-asso-dual-rep}, we have that $(V^*,\rho^*)$ is a representation of the 3-Lie algebra $(A,[\cdot,\cdot,\cdot])$, and $(V^*,-\mu^*)$ is a representation of the commutative associative algebra $(A,\cdot)$.	
		
		Next, we will prove that $\rho^*$ and $-\mu^*$ satisfy \eqref{rep-1} and \eqref{rep-2}. For all $x,y,z\in A, \xi\in V^*, v\in V$, we have
		\begin{eqnarray*}	
			&&\langle\big(\rho^*(x\cdot y,z)+\mu^*(x)\rho^*(y,z)+\mu^*(y)\rho^*(x,z)\big)\xi,v \rangle\\
			&\overset{\eqref{com-asso-dual},\eqref{3-lie-dual}}{=}&\langle \xi, \big(-\rho(x\cdot y,z)+\rho(y,z)\mu(x)+\rho(x,z)\mu(y)\big)v\rangle \\
			&=&\langle \xi, \big(-\rho(x\cdot y,z)+\mu([y,z,x])+\mu(x)\rho(y,z)+\mu([x,z,y])+\mu(y)\rho(x,z)\big)v\rangle\\
			&\overset{\eqref{rep-2}}{=}&\langle \xi, \big(-\rho(x\cdot y,z)+\rho(y,z)\mu(x)+\rho(x,z)\mu(y)\big)v\rangle\\
			&\overset{\eqref{rep-1}}{=}&0.
		\end{eqnarray*}
		Then 
		\begin{eqnarray*}
			\rho^*(x\cdot y,z)=-\mu^*(x)\rho^*(y,z)-\mu^*(y)\rho^*(x,z).
		\end{eqnarray*}
		Similarly, we have
		\begin{eqnarray*}	
			&&\langle\big(-\rho^*(x,y)\mu^*(z)+\mu^*([x,y,z])+\mu^*(z)\rho^*(x, y)\big)\xi,v \rangle  \\
			&\overset{\eqref{com-asso-dual},\eqref{3-lie-dual}}{=}&\langle\xi,\big(-\mu(z)\rho(x,y)-\mu([x,y,z])+\rho(x, y)\mu(z)\big)v \rangle\\
			&\overset{\eqref{rep-2}}{=}&0.
		\end{eqnarray*}
		Then 
		\begin{eqnarray*}
			-\rho^*(x,y)\mu^*(z)=-\mu^*([x,y,z])-\mu^*(z)\rho^*(x, y).
		\end{eqnarray*}
		
		Therefore, $(V,\rho^*,-\mu^*)$ is a representation of $(A,\cdot, [\cdot,\cdot,\cdot])$. In particular, taking $\rho=\mathrm{ad}$ and $\mu=\mathcal{L}$, we have that $(A^*,\mathrm{ad}^*,-\mathcal{L}^*)$ is a representation of $(A,\cdot, [\cdot,\cdot,\cdot])$. The conclusion holds.
	\end{proof}

	\begin{pro}\label{pro:possion-semi-direct}
		Let $(A,\cdot, [\cdot,\cdot,\cdot])$ be a Poisson 3-Lie algebra and $V$ be a vector space. Let $\mu:A\rightarrow \mathfrak{gl}(V)$ and $\rho:\otimes^2 A\rightarrow\mathfrak{gl}(V)$ be linear maps. Then $(V,\rho,\mu)$ is a representation of Poisson 3-Lie algebras $(A,\cdot, [\cdot,\cdot,\cdot])$ if and only if $(A\oplus V, \cdot_\mu, [\cdot,\cdot,\cdot]_\rho)$ is a  Poisson 3-Lie algebra, where the product $\cdot_{\mu}$ and the bracket $[\cdot,\cdot,\cdot]_{\rho}$ are given by
		\begin{align}
			(x_1+u_1)\cdot_{\mu}(x_2+u_2)&=x_1\cdot x_2+ \mu(x_1)u_2+\mu(x_2)u_1,\label{tp-semi-1}\\
			[x_1+u_1,x_2+u_2,x_3+u_3]_{\rho}&=[x_1,x_2,x_3]+\rho(x_1,x_2)u_3-\rho(x_1,x_3)u_2+\rho(x_2,x_3)u_1,\label{tp-semi-2}
		\end{align}
		for all $x_1,x_2,x_3\in A, u_1,u_2,u_3\in V$. We denote it by $A\ltimes_{\mu,\rho}V$, which is called \textbf{semi-direct product Poisson 3-Lie algebras}. 
	\end{pro}
	\begin{proof}
		Recall \cite{R. Schafer, Bai2019}, the semi-direct product of commutative associative algebras and 3-Lie algebras, we know that $(V,\mu)$ is a representation of a commutative associative algebra $(A,\cdot)$ if and only if  $(A\oplus V, \cdot_{\mu})$ is a commutative associative algebra, and $(V,\rho)$ is a representation of a 3-Lie algebra $(A,[\cdot,\cdot,\cdot])$ if and only if $(A\oplus V,[\cdot,\cdot,\cdot]_{\rho})$ is a  3-Lie algebra. 
		If $(V,\rho,\mu)$ is a representation of a Poisson 3-Lie algebra $(A,\cdot, [\cdot,\cdot,\cdot])$, then for all $x_i \in A, u_i \in V,1\leq i\leq4$, we have
		\begin{eqnarray*}
			&&[(x_1+u_1)\cdot_\mu(x_2+u_2),x_3+u_3,x_4+u_4]_\rho \\
			&\overset{\eqref{tp-semi-1}}{=}&[x_1\cdot x_2+ \mu(x_1)u_2+\mu(x_2)u_1,x_3+u_3,x_4+u_4]_\rho\\
			&\overset{\eqref{tp-semi-2}}{=}&[x_1\cdot x_2,x_3,x_4]+\rho(x_1\cdot x_2,x_3)u_4-\rho(x_1\cdot x_2,x_4)u_3+\rho(x_3,x_4)\big(\mu(x_1)u_2+\mu(x_2)u_1\big)\\
			&\overset{\eqref{rep-1},\eqref{rep-2}}{=}&x_1\cdot [x_2,x_3,x_4]+x_2\cdot [x_1,x_3,x_4]+\mu(x_1)\rho(x_2,x_3)u_4+\mu(x_2)\rho(x_1,x_3)u_4\\
			&&-\mu(x_1)\rho(x_2,x_4)u_3-\mu(x_2)\rho(x_1,x_4)u_3+\mu([x_3,x_4,x_1])u_2+\mu(x_1)\rho(x_3,x_4)u_2\\
			&&+\mu([x_3,x_4,x_2])u_1+\mu(x_2)\rho(x_3,x_4)u_1\\
			&\overset{\eqref{tp-semi-1}}{=}&(x_1+u_1)\cdot_{\mu}([x_2,x_3,x_4]+\rho(x_2,x_3)u_4-\rho(x_2,x_4)u_3+\rho(x_3,x_4)u_1) \\
			&&(x_2+u_2)\cdot_{\mu}([x_1,x_3,x_4]+\rho(x_1,x_3)u_4-\rho(x_1,x_4)u_3+\rho(x_3,x_4)u_1)\\
			&\overset{\eqref{tp-semi-2}}{=}&(x_1+u_1)\cdot_\mu[ x_2+u_2,x_3+u_3,x_4+u_4]_\rho+(x_2+u_2)\cdot_\mu[ x_1+u_1,x_3+u_3,x_4+u_4]_\rho.
		\end{eqnarray*}
		
		Therefore, $(A\oplus V, \cdot_\mu, [\cdot,\cdot,\cdot]_\rho)$ is a Poisson 3-Lie algebra. Conversely, in a similar manner, if $(A\oplus V, \cdot_\mu, [\cdot,\cdot,\cdot]_\rho)$ is a Poisson 3-Lie algebra, then $(V,\rho,\mu)$ is a representation of Poisson 3-Lie algebras. The conclusion holds.
	\end{proof}
	
	\subsection{Matched pairs of Poisson 3-Lie algebras}\label{subsec:42}
	\begin{defi}
		(\cite{Bai2019})	Let $(A,[\cdot,\cdot,\cdot]_A)$ and $(B,[\cdot,\cdot,\cdot]_B)$ be two 3-Lie algebras, and let $(B,\rho_A)$ and $(A,\rho_B)$ be representations of $(A,[\cdot,\cdot,\cdot]_A)$ and $(B,[\cdot,\cdot,\cdot]_B)$ respectively. If the following conditions are satisfied:
		{\footnotesize	\begin{align*}
				\rho_B(y_4,y_5)[x_1,x_2,x_3]_A &=[\rho_B(y_4,y_5)x_1,x_2,x_3]_A+[x_1,\rho_B(y_4,y_5)x_2,x_3]_A 
				+[x_1,x_2,\rho_B(y_4,y_5)x_3]_A,\\
				-\rho_B\big(\rho_A(x_1,x_2)y_3,y_5\big)x_4 &= -\rho_B\big(\rho_A(x_1,x_4)y_5,y_3\big)x_2+\rho_B\big(\rho_A(x_2,x_4)y_5,y_3\big)x_1-[x_1,x_2,\rho_B(y_3,y_5)x_4]_A,\\
				\left[ \rho_B(y_2,y_3)x_1,x_4,x_5\right]_A &= \rho_B(y_2,y_3)[x_1,x_4,x_5]_A+\rho_B\big(\rho_A(x_4,x_5)y_2,y_3\big)x_1+\rho_B\big(y_2,\rho_A(x_4,x_5)y_3\big)x_1,\\	
				\rho_A(x_4,x_5)[y_1,y_2,y_3]_B &=[\rho_A(x_4,x_5)y_1,y_2,y_3]_B+[y_1,\rho_A(x_4,x_5)y_2,y_3]_B +[y_1,y_2,\rho_A(x_4,x_5)y_3]_B,\\
				-\rho_A\big(\rho_B(y_1,y_2)x_3,x_5\big)y_4 &= -\rho_A\big(\rho_B(y_1,y_4)x_5,x_3\big)y_2+\rho_A\big(\rho_B(y_2,y_4)x_5,x_3\big)y_1-[y_1,y_2,\rho_A(x_3,x_5)y_4]_B,\\
				\left[ \rho_A(x_2,x_3)y_1,y_4,y_5\right]_B &= \rho_A(x_2,x_3)[y_1,y_4,y_5]_A+\rho_A\big(\rho_B(y_4,y_5)x_2,x_3\big)y_1+\rho_A\big(x_2,\rho_B(y_4,y_5)x_3\big)y_1,
		\end{align*}}
		for all $x_i\in A, y_i\in B, 1\leq i\leq 5$, then $(A,B,\rho_A,\rho_B)$ is called a \textbf{matched pair of 3-Lie algebras}.
	\end{defi}
	
	\begin{pro}\label{pro:direct-3-Lie}
		$($\cite{Bai2019}$)$ Let $(A,[\cdot,\cdot,\cdot]_A)$ and $(B,[\cdot,\cdot,\cdot]_B)$ be two 3-Lie algebras, and let  $\rho_A:\otimes^2 A\rightarrow \mathfrak{gl}(B)$, $\rho_B:\otimes^2 B\rightarrow \mathfrak{gl}(A)$ be two linear maps. Then for all $x_1,x_2,x_3\in A$, $y_1,y_2,y_3\in B$ there is a 3-Lie algebra structure on the vector space $A\oplus B$ defined by 
		\begin{eqnarray}\label{3-lie-sum}
			&&[x_1+y_1,x_2+y_2,x_3+y_3]_{A\oplus B}=[x_1,x_2,x_3]_A+\rho_A(x_1,x_2)y_3+\rho_A(x_3,x_1)y_2+\rho_A(x_2,x_3)y_1\nonumber\\
			&&+[y_1,y_2,y_3]_B+\rho_B(y_1,y_2)x_3+\rho_B(y_3,y_1)x_2+\rho_B(y_2,y_3)x_1, 	 
		\end{eqnarray}
		if and only if $(A,B,\rho_A,\rho_B)$ a matched pair of 3-Lie algebras. The 3-Lie algebra structure on the vector space $A\oplus B$ is denoted by $A\Join_{\rho_A}^{\rho_B} B$.
	\end{pro}
	
	\begin{defi} \label{def: MP-CAA}
		(\cite{Bai2010}) Let $(A,\cdot_A)$ and $(B,\cdot_B)$ be two commutative associative algebras, and $(B,\mu_A)$ and $(A,\mu_B)$ be representations of $(A,\cdot_A)$ and $(B,\cdot_B)$ respectively. If the following conditions are satisfied:
		\begin{align*}
			\mu_A(x)(a\cdot_B b)&=\big(\mu_A(x)a\big)\cdot_B b+ \mu_A\big(\mu_B(a)x\big)b, \\
			\mu_B(a)(x\cdot_A y)&=\big(\mu_B(a)x\big)\cdot_A y+ \mu_B\big(\mu_A(x)a\big)y,		
		\end{align*}
		for all $x,y\in A, a,b\in B$, then $(A,B,\mu_A,\mu_B)$ is called a \textbf{matched pair of commutative associative algebras}.
	\end{defi}
	
	\begin{pro}\label{pro:direct-comm-asso}
		$($\cite{Bai2010}$)$	Let  $(A,\cdot_A)$ and $(B,\cdot_B)$ be two commutative associative algebras, and $\mu_A: A\rightarrow\mathfrak{gl}(B)$, $\mu_B:B\rightarrow\mathfrak{gl}(A)$ be  two linear maps. Then there is a commutative associative algebra structure on $A\oplus B$ defined by 
		\begin{equation}\label{com-asso-sum}
			(x+a)\cdot(y+b)=x\cdot_A y+\mu_B(a)y+\mu_B(b)x+a\cdot_B b+\mu_A(x)b+\mu_A(y)a,\quad\forall~ x,y\in A,a,b\in B,
		\end{equation}
		if and only if $(A,B,\mu_A,\mu_B)$ is a matched pair of commutative associative algebras. The commutative associative algebra structure on $A\oplus B$ is denoted by $A\Join_{\mu_A}^{\mu_B} B$.
	\end{pro}
	\begin{defi}
		Let $(A,\cdot_A,[\cdot,\cdot,\cdot]_A)$ and $(B,\cdot_B,[\cdot,\cdot,\cdot]_B)$ be two Poisson 3-Lie algebras. Let $\mu_A:A\rightarrow \mathfrak{gl}(B)$, $\rho_A:\otimes^2 A\rightarrow \mathfrak{gl}(B)$ and $\mu_B:B\rightarrow \mathfrak{gl}(A)$, $\rho_B:\otimes^2 B\rightarrow \mathfrak{gl}(A)$ be linear maps such that $(A,B,\mu_A,\mu_B)$ is a matched pair of commutative associative algebra and $(A,B,\rho_A,\rho_B)$ is a matched pair of 3-Lie algebra. If $(B,\rho_A,\mu_A)$ and $(A,\rho_B,\mu_B)$ are  representations of $(A,\cdot_A,[\cdot,\cdot,\cdot]_A)$ and $(B,\cdot_B,[\cdot,\cdot,\cdot]_B)$ respectively, and for all $x_i\in A, y_i\in B, 1\leq i\leq 4$, the following equations hold:
		\begin{eqnarray}
			&&[\mu_B(y_1)x_1,x_2,x_3]=\mu_B\big(\rho_A(x_2,x_3)y_1\big)x_1+\mu_B(y_1)([x_1,x_2,x_3]),\label{cond-1}\\
			&&[\mu_A(x_1)y_1,y_2,y_3]=\mu_A\big(\rho_B(y_2,y_3)x_1\big)y_1+\mu_A(x_1)([y_1,y_2,y_3]),\label{cond-2}\\
			&&\rho_B\big(\mu_A(x_1)y_1,y_2\big)x_2=x_1\cdot_{A}\big(\rho_B(y_1,y_2)x_2\big),\label{cond-3}\\
			&&\rho_A\big(\mu_B(y_1)x_1,x_2\big)y_2=y_1\cdot_{B}\big(\rho_A(x_1,x_2)y_2\big),\label{cond-4}\\
			&&\mu_B\big(\rho_A(x_1,x_2)y_1\big)x_3+\mu_B\big(\rho_A(x_3,x_2)y_1\big)x_1=0,\label{cond-5}\\
			&&\mu_A\big(\rho_B(y_1,y_2)x_1\big)y_3+\mu_A\big(\rho_B(y_3,y_2)x_1\big)y_1=0,\label{cond-6}\\
			&&\rho_B(y_3,y_4)(x_1\cdot_{A}x_2)=x_1\cdot_{A}\big(\rho_B(y_3,y_4)x_2\big)+x_2\cdot_{A}\big(\rho_B(y_3,y_4)x_1\big),\label{cond-7}\\	
			&&\rho_A(x_3,x_4)(y_1\cdot_{B}y_2)=y_1\cdot_{B}\big(\rho_A(x_3,x_4)y_2\big)+y_2\cdot_{B}\big(\rho_A(x_3,x_4)y_1\big),\label{cond-8}
		\end{eqnarray}
		then $(A,B,\rho_A,\mu_A,\rho_B,\mu_B)$ is called a \textbf{matched pair of Poisson 3-Lie algebras}.
	\end{defi}
	\begin{thm}\label{matched-pair}
		Let $(A,\cdot_A,[\cdot,\cdot,\cdot]_A)$ and $(B,\cdot_B,[\cdot,\cdot,\cdot]_B)$ be two Poisson 3-Lie algebras. Let $\mu_A:A\rightarrow \mathfrak{gl}(B)$, $\rho_A:\otimes^2 A\rightarrow \mathfrak{gl}(B)$ and $\mu_B:B\rightarrow \mathfrak{gl}(A)$, $\rho_B:\otimes^2 B\rightarrow \mathfrak{gl}(A)$ are linear maps. Define two linear operations $\cdot$ and $[\cdot,\cdot,\cdot]$ on $A\oplus B$ by \eqref{com-asso-sum} and \eqref{3-lie-sum}, respectively. Then $(A\oplus B,\cdot,[\cdot,\cdot,\cdot])$ is a Poisson 3-Lie algebra if and only if $(A,B,\rho_A,\mu_A,\rho_B,\mu_B)$ is a matched pair of Poisson 3-Lie algebras. We denote it by $A\Join _{\mu_A,\rho_A}^{\mu_B,\rho_B}B$.
	\end{thm}
	\begin{proof} 
		By Proposition \ref{pro:direct-3-Lie} and Proposition \ref{pro:direct-comm-asso},  $(A\oplus B,\cdot,[\cdot,\cdot,\cdot])$ is a Poisson 3-Lie algebra if and only if for all $x_i\in A, y_i\in B, 1\leq i\leq 4$,
		\begin{eqnarray*}
			\underbrace{[(x_1+y_1)\cdot(x_2+y_2),x_3+y_3,x_4+y_4]}_{(A1)}&=&\underbrace{(x_1+y_1)\cdot[x_2+y_2,x_3+y_3,x_4+y_4]}_{(A2)}\\
			&&+\underbrace{(x_2+y_2)\cdot[x_1+y_1,x_3+y_3,x_4+y_4]}_{(A3)}.
		\end{eqnarray*}	
		By \eqref{com-asso-sum} and \eqref{3-lie-sum}, we have
		\begin{eqnarray*}
			(A1)&=&[x_1\cdot_{A}x_2+\mu_B(y_2)x_1+\mu_B(y_1)x_2+y_1\cdot_{B}y_2+\mu_A(x_2)y_1+\mu_A(x_1)y_2,x_3+y_3,x_4+y_4]  \\
			&=&\underbrace{[x_1\cdot_{A}x_2,x_3,x_4]}_{(a1)}+\underbrace{[\mu_B(y_1)x_2,x_3,x_4]}_{(d1)}+\underbrace{[\mu_B(y_2)x_1,x_3,x_4]}_{(d4)}\\
			&&+\underbrace{\rho_B(y_1\cdot_{B}y_2,y_3)x_4}_{(b1)}+\underbrace{\rho_B\big(\mu_A(x_2)y_1,y_3\big)x_4}_{(f1)}+\underbrace{\rho_B\big(\mu_A(x_1)y_2,y_3\big)x_4}_{(f5)}\\
			&&+\underbrace{\rho_B\big(y_4,y_1\cdot_{B}y_2\big)x_3}_{(b4)}+\underbrace{\rho_B\big(y_4,\mu_A(x_1)y_2\big)x_3}_{(f9)}+\underbrace{\rho_B\big(y_4,\mu_A(x_2)y_1\big)x_3}_{(f13)}\\
			&&+\underbrace{\rho_B(y_3,y_4)(x_1\cdot_{A}x_2)}_{(j1)}+\underbrace{\rho_B(y_3,y_4)\big(\mu_B(y_2)x_1\big)}_{(c1)}+\underbrace{\rho_B(y_3,y_4)\big(\mu_B(y_1)x_2\big)}_{(c4)}
			\\
			&&+\underbrace{[y_1\cdot_{B}y_2,y_3,y_4]}_{(a4)}+\underbrace{[\mu_A(x_2)y_1,y_3,y_4]}_{(e4)}+\underbrace{[\mu_A(x_1)y_2,y_3,y_4]}_{(e1)}\\
			&&+\underbrace{\rho_A(x_1\cdot_{A}x_2,x_3)y_4}_{(b7)}+\underbrace{\rho_A\big(\mu_B(y_2)x_1,x_3\big)y_4}_{(f10)}+\underbrace{\rho_A\big(\mu_B(y_1)x_2,x_3\big)y_4}_{(f14)}\\
			&&+\underbrace{\rho_A(x_4,x_1\cdot_{A}x_2)y_3}_{(b10)}+\underbrace{\rho_A\big(x_4,\mu_B(y_2)x_1\big)y_3}_{(f6)}+\underbrace{\rho_A\big(x_4,\mu_B(y_1)x_2\big)y_3}_{(f2)}\\
			&&+\underbrace{\rho_A(x_3,x_4)(y_1\cdot_{B}y_2)}_{(k1)}+\underbrace{\rho_A(x_3,x_4)\big(\mu_A(x_2)y_1\big)}_{(c7)}+\underbrace{\rho_A(x_3,x_4)\big(\mu_A(x_1)y_2\big)}_{c10},\\
			(A2)&=&(x_1+y_1)\cdot\big([x_2,x_3,x_4]+\rho_B(y_2,y_3)x_4+\rho_B(y_4,y_2)x_3+\rho_B(y_3,y_4)x_2\\
			&&+[y_2,y_3,y_4]+\rho_A(x_2,x_3)y_4+\rho_A(x_4,x_2)y_3+\rho_A(x_3,x_4)y_2\big)\\
			&=&\underbrace{x_1\cdot_{A}[x_2,x_3,x_4]}_{(a2)}+\underbrace{x_1\cdot_{A}\rho_B(y_2,y_3)x_4}_{(f7)}+\underbrace{x_1\cdot_{A}\rho_B(y_4,y_2)x_3}_{(f11)}+\underbrace{x_1\cdot_{A}\rho_B(y_3,y_4)x_2}_{(j2)}\\
			&&+\underbrace{\mu_B([y_2,y_3,y_4])x_1}_{(c2)}+\underbrace{\mu_B\big(\rho_A(x_2,x_3)y_4\big)x_1}_{(g1)}+\underbrace{\mu_B\big(\rho_A(x_4,x_2)y_3\big)x_1}_{(g3)}\\
			&&+\underbrace{\mu_B\big(\rho_A(x_3,x_4)y_2\big)x_1}_{(d5)}+\underbrace{\mu_B(y_1)([x_2,x_3,x_4])}_{(c8)}+\underbrace{\mu_B(y_1)\big(\rho_B(y_2,y_3)x_4\big)}_{(b2)}\\
			&&+\underbrace{\mu_B(y_1)\big(\rho_B(y_4,y_2)x_3\big)}_{(b5)}+\underbrace{\mu_B(y_1)\big(\rho_B(y_3,y_4)x_2\big)}_{(c5)}+\underbrace{y_1\cdot_{B}[y_2,y_3,y_4]}_{(a5)}\\
			&&+\underbrace{y_1\cdot_{B}\big(\rho_A(x_2,x_3)y_4\big)}_{(f15)}+\underbrace{y_1\cdot_{B}\big(\rho_A(x_4,x_2)y_3\big)}_{(f3)}+\underbrace{y_1\cdot_{B}\big(\rho_A(x_3,x_4)y_2\big)}_{(k2)}\\
			&&+\underbrace{\mu_A([x_2,x_3,x_4])y_1}_{(d2)}+\underbrace{\mu_A\big(\rho_B(y_2,y_3)x_4\big)y_1}_{(h1)}+\underbrace{\mu_A\big(\rho_B(y_4,y_2)x_3\big)y_1}_{(h3)}\\
			&&+\underbrace{\mu_A\big(\rho_B(y_3,y_4)x_2\big)y_1}_{(e5)}+\underbrace{\mu_A(x_1)([y_2,y_3,y_4])}_{(e2)}+\underbrace{\mu_A(x_1)\big(\rho_A(x_2,x_3)y_4\big)}_{(b8)}\\
			&&+\underbrace{\mu_A(x_1)(\rho_A(x_4,x_2)y_3\big)}_{(b11)}+\underbrace{\mu_A(x_1)\big(\rho_A(x_3,x_4)y_2\big)}_{(c11)},\\
			(A3)&=&(x_2+y_2)\cdot\big([x_1,x_3,x_4]+\rho_B(y_1,y_3)x_4+\rho_B(y_4,y_1)x_3+\rho_B(y_3,y_4)x_1\\
			&&+[y_1,y_3,y_4]+\rho_A(x_1,x_3)y_4+\rho_A(x_4,x_1)y_3+\rho_A(x_3,x_4)y_1\big)\\
			&=&\underbrace{x_2\cdot_{A}[x_1,x_3,x_4]}_{(a3)}+\underbrace{x_2\cdot_{A}\rho_B(y_1,y_3)x_4}_{(f4)}+\underbrace{x_2\cdot_{A}\rho_B(y_4,y_1)x_3}_{(f16)}+\underbrace{x_2\cdot_{A}\rho_B(y_3,y_4)x_1}_{(j3)}\\
			&&+\underbrace{\mu_B([y_1,y_3,y_4])x_2}_{(c6)}+\underbrace{\mu_B\big(\rho_A(x_1,x_3)y_4\big)x_2}_{(g2)}+\underbrace{\mu_B\big(\rho_A(x_4,x_1)y_3\big)x_2}_{(g4)}\\
			&&+\underbrace{\mu_B\big(\rho_A(x_3,x_4)y_1\big)x_2}_{(d3)}+\underbrace{\mu_B(y_2)([x_1,x_3,x_4])}_{(d6)}+\underbrace{\mu_B(y_2)\big(\rho_B(y_1,y_3)x_4\big)}_{(b3)}\\
			&&+\underbrace{\mu_B(y_2)\big(\rho_B(y_4,y_1)x_3\big)}_{(b6)}+\underbrace{\mu_B(y_2)\big(\rho_B(y_3,y_4)x_1\big)}_{(c3)}+\underbrace{y_2\cdot_{B}[y_1,y_3,y_4]}_{(a6)}\\
			&&+\underbrace{y_2\cdot_{B}\big(\rho_A(x_1,x_3)y_4\big)}_{(f12)}+\underbrace{y_2\cdot_{B}\big(\rho_A(x_4,x_1)y_3\big)}_{(f8)}+\underbrace{y_2\cdot_{B}\big(\rho_A(x_3,x_4)y_1\big)}_{(k3)}\\
			&&+\underbrace{\mu_A([x_1,x_3,x_4])y_2}_{(c12)}+\underbrace{\mu_A\big(\rho_B(y_1,y_3)x_4\big)y_2}_{(h2)}+\underbrace{\mu_A\big(\rho_B(y_4,y_1)x_3\big)y_2}_{(h4)}\\
			&&+\underbrace{\mu_A\big(\rho_B(y_3,y_4)x_1\big)y_2}_{(e3)}+\underbrace{\mu_A(x_2)([y_1,y_3,y_4])}_{(e6)}+\underbrace{\mu_A(x_2)\big(\rho_A(x_1,x_3)y_4\big)}_{(b9)}\\
			&&+\underbrace{\mu_A(x_2)\big(\rho_A(x_4,x_1)y_3\big)}_{(b12)}+\underbrace{\mu_A(x_2)\big(\rho_A(x_3,x_4)y_1\big)}_{(c9)}.	
		\end{eqnarray*}
		By \eqref{eq:poisson-alg}, we have $$(a1)=(a2)+(a3),~ (a4)=(a5)+(a6).$$
		By \eqref{rep-1}, we have $$(b1)=(b2)+(b3), ~(b4)=(b5)+(b6), ~(b7)=(b8)+(b9),~ (b10)=(b11)+(b12).$$ 
		By \eqref{rep-2}, we have $$(c1)=(c2)+(c3), ~(c4)=(c5)+(c6), ~(c7)=(c8)+(c9),~ (c10)=(c11)+(c12).$$
		Moreover, we have the following equivalences
		\begin{eqnarray*}
			(d1)=(d2)+(d3), ~(d4)=(d5)+(d6) &\Leftrightarrow& \eqref{cond-1};\\
			(e1)=(e2)+(e3), ~(e4)=(e5)+(e6) &\Leftrightarrow& \eqref{cond-2};\\
			(f1)+(f2)=(f3)+(f4),~ (f5)+(f6)=(f7)+(f8) &\Leftrightarrow& \eqref{cond-3},\eqref{cond-4}; \\
			(f9)+(f10)=(f11)+(f12),~ (f13)+(f14)=(f15)+(f16)  &\Leftrightarrow& \eqref{cond-3},\eqref{cond-4};\\
			(g1)+(g2)=0, ~(g3)+(g4)=0 	&\Leftrightarrow& \eqref{cond-5};\\
			(h1)+(h2)=0,~ (h3)+(h4)=0 &\Leftrightarrow& \eqref{cond-6};\\
			(j1)=(j2)+(j3) &\Leftrightarrow& \eqref{cond-7};\\
			(k1)=(k2)+(k3) &\Leftrightarrow& \eqref{cond-8}.
		\end{eqnarray*}
		Therefore, $(A,B,\rho_A,\mu_A,\rho_B,\mu_B)$ is a matched pair of Poisson 3-Lie algebras if and only if $(A\oplus B,\cdot,[\cdot,\cdot,\cdot])$ is a Poisson 3-Lie algebra. The proof is completed.	
	\end{proof}
	\subsection{Manin triples of Poisson 3-Lie algebras}\label{subsec:43}
	\begin{defi}
		Let $(A,\cdot,[\cdot,\cdot,\cdot])$ be a Poisson 3-Lie algebra and $B$ be a subspace of $A$. If for all $x,y,z\in B$ ,	$x\cdot y \in B$ and $[x,y,z] \in B$, then $(B,\cdot,[\cdot,\cdot,\cdot])$ is called a \textbf{Poisson 3-Lie subalgebra} of $(A,\cdot,[\cdot,\cdot,\cdot])$.
	\end{defi}
	\begin{defi}
		Let $(A,\cdot,[\cdot,\cdot,\cdot])$ be a Poisson 3-Lie algebra. A bilinear form $\mathcal{B}:A\otimes A\rightarrow \mathbb{F}$  is called \textbf{invariant }if it satisfies
		\begin{equation*}
			\mathcal{B}(x\cdot y,z)=\mathcal{B}(x,y\cdot z), \quad  \mathcal{B}([x,y,z],u)=-\mathcal{B}([x,y,u],z), \quad \forall~ x,y,z,u\in A.
		\end{equation*}
		A Poisson 3-Lie algebra $(A,\cdot,[\cdot,\cdot,\cdot])$ is called \textbf{pseudo-metric} if there is a nondegenerate symmetric invariant bilinear form on $(A,\cdot,[\cdot,\cdot,\cdot])$. We denote it by $(A,\mathcal{B})$.
	\end{defi}
	The Manin triples of Poisson 3-Lie algebras should be both  Manin triples of 3-Lie algebras \cite{Bai2019} and  Manin triples of commutative associative algebras \cite{Bai2010}. In the following, we give the definition of Manin triples of Poisson 3-Lie algebras. For simplicity, the Poisson 3-Lie algebra $(A,\cdot,[\cdot,\cdot,\cdot])$ is denoted by $A$.
	\begin{defi}\label{Manin-triple}
		A \textbf{Manin triple of Poisson 3-Lie algebras} is a triple $\big((A,\mathcal{B}), A_1, A_2\big)$ consisting of a pseudo-metric Poisson 3-Lie algebra $(A,\mathcal{B})$, and two Poisson 3-Lie algebras $A_1$ and $A_2$  that satisfying:
		
		(a) $A_1$ and $A_2$ are Poisson 3-Lie subalgebras of $A$ that are isotropic with respect to $\mathcal{B}$, that is, $\mathcal{B}(x_1, y_1)=\mathcal{B}(x_2, y_2)=0$, for all $x_1, y_1\in A_1$ and $x_2, y_2\in A_2$;
		
		(b) $A = A_1 \oplus A_2$ as the direct sum of vector spaces;
		
		(c) For all $x_1, y_1 \in A_1$ and $x_2, y_2\in A_2$,  $\mathrm{pr}_1[x_1, y_1, x_2] = 0$ and $\mathrm{pr}_2[x_2, y_2, x_1] = 0$, where $\mathrm{pr}_1$ and $\mathrm{pr}_2$ are the projections from $A_1 \oplus A_2$ to $A_1$ and $A_2$ respectively.
	\end{defi}
	\begin{defi}
		A \textbf{homomorphism} between two Manin triples of Poisson 3-Lie algebras $\big((A,\mathcal{B}_1), A_1, A_2\big)$ and $\big((A',\mathcal{B}_2), A'_1, A'_2\big)$ is a homomorphism $f: A \rightarrow A'$ of Poisson 3-Lie algebras satisfying
		\begin{equation*}
			f(A_1)\subset A'_1, \quad f(A_2)\subset A'_2, \quad \mathcal{B}_1(x_1, x_2)=\mathcal{B}_2(f(x_1),f(x_2)), \quad \forall~ x_1, x_2\in A.
		\end{equation*}
	\end{defi}
	Let $(A,\cdot, [\cdot,\cdot,\cdot])$ and $(A^*,\circ,[\cdot,\cdot,\cdot]^*)$ be two Poisson 3-Lie algebras. On $A\oplus A^*$, there is a natural nondegenerate symmetric bilinear form $\mathcal{B}$ given by
	\begin{equation}\label{bilinear}
		\mathcal{B}(x+\xi, y+\eta)=\langle x,\eta\rangle + \langle y,\xi\rangle, \quad \forall~ x, y \in A, \xi, \eta \in A^*.
	\end{equation}
	Then $A$ and $A^*$ are isotropic with respect to $\mathcal{B}$, that is, $\mathcal{B}(x,y)=\mathcal{B}(\xi,\eta)=0$.
	
	Define the operations $[\cdot,\cdot,\cdot]_{A\oplus A^*}:\otimes^3(A\oplus A^*)\rightarrow A\oplus A^*$ and $\cdot_{A\oplus A^*}:\otimes^2(A\oplus A^*)\rightarrow A\oplus A^*$  by 
	\begin{eqnarray}
		[x+\xi,y+\eta,z+\zeta]_{A\oplus A^*}&=&[x,y,z]+\mathrm{ad}^*_{x,y}\zeta-\mathrm{ad}^*_{x,z}\eta+\mathrm{ad}^*_{y,z}\xi
		\nonumber \\
		&&+[\xi,\eta,\zeta]^*+\mathfrak{ad}^*_{\xi,\eta}z-\mathfrak{ad}^*_{\xi,\zeta}y+\mathfrak{ad}^*_{\eta,\zeta}x,\label{Manin-1} \\
		(x+\xi)\cdot_{A\oplus A^*}(y+\eta)&=&x\cdot y -\mathcal{L}^*(x)\eta-\mathcal{L}^*(x)\xi+\xi\circ\eta-L^*(\xi) y-L^*(\eta) x,\label{Manin-2}
	\end{eqnarray}
	for all $x,y,z\in A, \xi,\eta,\zeta\in A^*$, where $(A^*,\mathrm{ad}^*,-\mathcal{L}^*)$ and $(A,\mathfrak{ad}^*,-L^*)$ are the coadjoint representation of $(A,\cdot,[\cdot,\cdot,\cdot])$ and $(A^*,\circ,[\cdot,\cdot,\cdot]^*)$, respectively. 
	\begin{lem}\label{lem:Manin}
		With the above notations, if $(A\oplus A^*,\cdot_{A\oplus A^*},[\cdot,\cdot,\cdot]_{A\oplus A^*})$ is a Poisson 3-Lie algebra, then $\big((A\oplus A^*,\mathcal{B}), A, A^*\big)$ is a Manin triple of Poisson 3-Lie algebras. We call it  \textbf{the standard Manin triple of Poisson 3-Lie algebras}.
	\end{lem}
	\begin{proof}
		By \eqref{bilinear}, \eqref{Manin-1}, and \eqref{Manin-2}, for all $x,y,z,u\in A,~ \xi,\eta,\zeta,\gamma\in A^*$, we have $$\mathcal{B}([x+\xi,y+\eta,z+\zeta]_{A\oplus A^*},u+\gamma)=-\mathcal{B}([x+\xi,y+\eta,u+\gamma]_{A\oplus A^*},z+\zeta),$$
		$$\mathcal{B}\big((x+\xi)\cdot_{A\oplus A^*}(y+\eta),z+\zeta\big)=\mathcal{B}\big(x+\xi,(y+\eta)\cdot_{A\oplus A^*}(z+\zeta)\big).$$  Then $(A\oplus A^*,\mathcal{B})$ is a  pseudo-metric Poisson 3-Lie algebra. It is straightforward to see that $A$ and $A^*$ are Poisson 3-Lie subalgebras of $A\oplus A^*$, and they are isotropic with respect to $\mathcal{B}$.  Moreover, the bracket $[\cdot,\cdot,\cdot]_{A\oplus A^*}$ satisfies (c) in Definition \ref{Manin-triple}. Therefore, $\big((A\oplus A^*,\mathcal{B}), A, A^*\big)$ is a Manin triple of Poisson 3-Lie algebras.
	\end{proof}
	\begin{thm}\label{Manin triple}
		Let $(A, \cdot,[\cdot,\cdot,\cdot])$ and $(A^*,\circ,[\cdot,\cdot,\cdot]^*)$ be two Poisson 3-Lie algebras. Then $\big((A\oplus A^*,\mathcal{B}), A, A^*\big)$ is a standard Manin triple of Poisson 3-Lie algebras if and only if \\$(A, A^*, \mathrm{ad}^*,-\mathcal{L}^*,\mathfrak{ad}^*,-L^*)$ is a matched pair of Poisson 3-Lie algebras.
	\end{thm}
	\begin{proof}
		Combining Lemma \ref{lem:Manin} and Theorem \ref{matched-pair}, we can get the conclusion directly.
	\end{proof}
	\subsection{Double construction Poisson 3-Lie bialgebras}\label{subsec:44}
	\begin{defi}
		(\cite{X. Ni}) A \textbf{cocommutative coassociative coalgebra} is a pair $(A, \Delta)$, where $A$ is a vector space  and $\Delta:A\rightarrow A\otimes A$ is a linear map satisfying
		\begin{equation}\label{eq:com-asso-coal}
			\tau\Delta(x)=\Delta(x), \quad (\Delta\otimes 1)\Delta(x)=(1\otimes\Delta)\Delta(x), \quad \forall~x\in A,
		\end{equation}
		where $\tau:A\otimes A\rightarrow A\otimes A$, $x\otimes y\mapsto y\otimes x$.
	\end{defi}
	\begin{pro}\label{pro:commu-asso-coal}
		$($\cite{X. Ni}$)$	Let $A$ be a vector space and $\Delta:A\rightarrow A\otimes A$ be a linear map. If $\Delta^*:A^*\otimes A^*\rightarrow A^*$ is a linear map satisfying $$\langle \Delta^{*}(\xi\otimes \eta),x\rangle=\langle \xi\otimes \eta, \Delta(x)\rangle,\quad \forall~\xi, \eta \in A^*, ~x\in A,$$ 
		then $(A,\Delta)$ is a cocommutative coassociative coalgebra if and only if $(A^*,\Delta^*)$ is a commutative associative algebra.
	\end{pro}
	\begin{defi}
		(\cite{Bai2010}) A \textbf{commutative and cocommutative infinitesimal bialgebra} is a triple $(A,\cdot,\Delta)$, where $(A,\cdot)$ is a commutative associative algebra and $(A,\Delta)$ is a cocommutative coassociative coalgebra satisfying
		\begin{equation}\label{eq:infi-bialg}
			\Delta(x\cdot y)=(\mathcal{L}_x\otimes 1)\Delta(y)+(1\otimes \mathcal{L}_y)\Delta(x), \quad \forall~ x,y \in A,
		\end{equation} 
		where $\mathcal{L}_x$  denotes the left multiplication of $x$ and $1$ denotes the identity map.
	\end{defi}
	Let $A$ be a vector space. For any $r=x_1\otimes x_2\otimes \cdots\otimes x_n\in \otimes^n A$, $x_i\in A$, $1\leq i\leq n$. Define the switching operator $\tau_{ij}:\otimes^n A\rightarrow\otimes^n A$ as follows:
	\begin{equation*}
		\tau_{ij}(r)=x_1\otimes\cdots\otimes x_j\otimes \cdots\otimes x_i\otimes\cdots\otimes x_n, \quad 1\leq i<j\leq n.
	\end{equation*}
	%Let $(A,[\cdot,\cdot,\cdot])$ be a 3-Lie algebra,
	\begin{pro}\label{pro:3-Lie-coal}
		$($\cite{Bai2019}$)$	Let $A$ be a vector space and $\delta:A\rightarrow A\otimes A\otimes A$ be a linear map. If $\delta^*:A^*\otimes A^*\otimes A^*\rightarrow A^*$ is a linear map satisfying $$\langle \delta^{*}(\xi_1\otimes \xi_2\otimes\xi_3),x\rangle=\langle \xi_1\otimes \xi_2\otimes\xi_3, \delta(x)\rangle,\quad \forall~\xi_1, \xi_2, \xi_3 \in A^*, ~x\in A,$$ 
		then $\delta^*$ defines a 3-Lie algebra on $A^*$ if and only if $\delta^*$ is a skew-symmetric operation and  $\delta$ satisfies
		\begin{equation}\label{eq:3-Lie-coal}
			(\delta\otimes 1\otimes 1)\delta(x)+\tau_{23}\tau_{12}\big((1\otimes\delta\otimes 1)\delta(x)\big)+\tau_{13} \tau_{24}\big((1\otimes 1\otimes\delta)\delta(x)\big)=(1\otimes 1\otimes\delta)\delta(x).
		\end{equation}
	\end{pro}
	\begin{defi}
		(\cite{R. Bai})	A \textbf{3-Lie coalgebra} is a pair $(A,\delta)$, where $A$ is a vector space  and $\delta:A\rightarrow A\otimes A\otimes A$ is a linear map satisfying \eqref{eq:3-Lie-coal}.
	\end{defi}
	\begin{defi}
		A \textbf{Poisson 3-Lie coalgebra} is a triple $(A,\Delta,\delta)$, where $(A,\delta)$ is a 3-Lie coalgebra and $(A,\Delta)$ is a cocommutative coassociative coalgebra that satisfies
		\begin{equation}\label{Poisson-coalg}
			(\Delta\otimes 1\otimes 1)\delta(x)=(1\otimes \delta)\Delta(x)+(\tau_{12}\otimes 1\otimes 1)(1\otimes\delta)\Delta(x),\quad\forall~ x\in A.
		\end{equation}
	\end{defi}
	\begin{pro}\label{pro:poisson-3-lie-coal}
		Let $A$ be a vector space, $(A,\delta)$ be a 3-Lie coalgebra and $(A,\Delta)$ be a cocommutative coassociative coalgebra. Then $(A,\Delta,\delta)$ is a Poisson 3-Lie coalgebra if and only if $(A^*,\Delta^*,\delta^*)$ is a Poisson 3-Lie algebra.
	\end{pro}
	\begin{proof}
		By Proposition \ref{pro:commu-asso-coal} and Proposition \ref{pro:3-Lie-coal}, we have $(A^*,\Delta^*)$ is a commutative associative algebra and $(A^*,\delta^*)$ is a 3-Lie algebra. Next we prove that $\Delta^*$ and $\delta^*$ satisfy \eqref{eq:poisson-alg}. For all $x\in A, a^*,b^*,c^*,d^* \in A^*$, we have
		\begin{eqnarray*}
			&&\langle (\Delta\otimes 1\otimes 1)\delta(x)-(1\otimes \delta)\Delta(x)-(\tau_{12}\otimes 1\otimes 1)(1\otimes\delta)\Delta(x), a^*\otimes b^*\otimes c^*\otimes d^* \rangle \\
			&=&\langle x, \delta^*(\Delta^*\otimes 1\otimes 1)(a^*\otimes b^*\otimes c^*\otimes d^*) \rangle -\langle x, \Delta^*(1\otimes \delta^*)(a^*\otimes b^*\otimes c^*\otimes d^*) \rangle\\
			&&-\langle x, \Delta^*(1\otimes \delta^*)(b^*\otimes a^*\otimes c^*\otimes d^*) \rangle\\
			&=&\langle x, \delta^*\big(\Delta^*(a^*\otimes b^*)\otimes c^*\otimes d^*\big)-\Delta^*\big(a^*\otimes \delta^*(b^*\otimes c^*\otimes d^*)\big) -\Delta^*\big(b^*\otimes \delta^*(a^*\otimes c^*\otimes d^*)\big) \rangle\\
			&=&0  ,
		\end{eqnarray*}
		Thus \eqref{Poisson-coalg} holds if and only if
		\begin{eqnarray*}
			\delta^*\big(\Delta^*(a^*\otimes b^*)\otimes c^*\otimes d^*\big)=\Delta^*\big(a^*\otimes \delta^*(b^*\otimes c^*\otimes d^*)\big) +\Delta^*\big(b^*\otimes\delta^*(a^*\otimes c^*\otimes d^*)\big).
		\end{eqnarray*}
		Therefore, $(A,\Delta,\delta)$ is a Poisson 3-Lie coalgebra if and only if $(A^*,\Delta^*,\delta^*)$ is a Poisson 3-Lie algebra. The proof is completed.
	\end{proof}
	\begin{lem}\label{equival}
		Let $(A,\cdot, [\cdot,\cdot,\cdot])$ and $(A^*,\circ,[\cdot,\cdot,\cdot]^*)$ be two Poisson 3-Lie algebras. Then $(A, A^*, \mathrm{ad}^*,-\mathcal{L}^*,\mathfrak{ad}^*,-L^*)$ is a matched pair of Poisson 3-Lie algebras if and only if \eqref{cond-1}, \eqref{cond-2}, \eqref{cond-3} and \eqref{cond-4} hold for $\rho_A=\mathrm{ad}^*,\mu_A=-\mathcal{L}^*,\rho_B=\mathfrak{ad}^*,\mu_B=-L^*$.
	\end{lem}
	\begin{proof}
		Firstly, we prove \eqref{cond-1}$\Leftrightarrow$\eqref{cond-8} using the structure constants. Let $\{e_1,e_2,\cdots,e_n\}$ be a basis of the vector space $A$ and $\{e^*_1,e^*_2,\cdots,e^*_n\}$ be the corresponding dual basis. Let
		\begin{eqnarray*}
			&&e_i\cdot e_j=\sum_{l=1}^n{a_{ij}^l e_l}, ~ e^*_i\circ e^*_j=\sum_{l=1}^n{b_{ij}^l e^*_l},~
			[e_i,e_j,e_k]=\sum_{l=1}^n{c_{ijk}^l e_l},~ [e^*_i,e^*_j,e^*_k]^*=\sum_{l=1}^n{d_{ijk}^l e^*_l},
		\end{eqnarray*}
		where $a_{ij}^l,~b_{ij}^l,~c_{ijk}^l,~d_{ijk}^l\in \mathbb{F},~ i,~j,~k,~l=1,2,\cdots,n$. Then by the definitions of $\mathrm{ad}^*$, $\mathfrak{ad}^*$, $\mathcal{L}^*$, $L^*$,  $\delta$, and $\Delta $ we have 
		\begin{eqnarray*}
			\mathrm{ad}^*_{e_i,e_j}e_k^*=-\sum_{l=1}^n{c_{ijl}^k e_l^*}, \quad \mathfrak{ad}^*_{e_i^*,e_j^*}e_k=-\sum_{l=1}^n{d_{ijl}^k e_l},\quad \delta(e_k)=\sum_{i,j,l=1}^n{d_{ijl}^k e_i\otimes e_j\otimes e_l},\\
			\mathcal{L}^*(e_i)e^*_k=-\sum_{l=1}^n{a_{il}^k e_l^*},\quad L^*(e^*_i)e_k=-\sum_{l=1}^n{b_{il}^k e_l},\quad \Delta(e_k)=\sum_{i,j=1}^n{b_{ij}^k e_i\otimes e_j}.
		\end{eqnarray*}
		
		Let $\rho_A=\mathrm{ad}^*, ~\mu_A=-\mathcal{L}^*, ~\rho_B=\mathfrak{ad}^*$ and $\mu_B=-L^*$. For all $e_\alpha^*\in A^*, ~\alpha=1,2,\cdots,n$, by \eqref{cond-1}, we have
		\begin{eqnarray*}
			[-L^*(e^*_\alpha)e_i,e_j,e_k]=-L^*(\mathrm{ad}^*_{e_j,e_k}e^*_\alpha)e_i-L^*(e^*_\alpha)([e_i,e_j,e_k]).
		\end{eqnarray*}
		Then
		\begin{eqnarray*}
			[\sum_{l=1}^n{b_{\alpha l}^i e_l},e_j,e_k]=-L^*(-\sum_{l=1}^n{c_{jkl}^\alpha e_l^*})e_i-L^*(e^*_\alpha)(\sum_{l=1}^n{c_{ijk}^l e_l}).
		\end{eqnarray*}
		Furthermore, for all $m=1,2,\cdots,n$,
		\begin{eqnarray*}
			\sum_{l=1}^n{b_{\alpha l}^i}\sum_{m=1}^n{c_{ljk}^m e_m}=-\sum_{l=1}^n{c_{jkl}^\alpha}\sum_{m=1}^n{b_{lm}^i e_m}+\sum_{l=1}^n{c_{ijk}^l}\sum_{m=1}^n{b_{\alpha m}^l e_m}.
		\end{eqnarray*}		
		Thus
		\begin{equation}\label{eq:sture-con}
			\sum_{l=1}^n{(b_{\alpha l}^ic_{ljk}^m+c_{jkl}^\alpha b_{lm}^i-c_{ijk}^l b_{\alpha m}^l)}=0.
		\end{equation}
		
		For all $e^*_\alpha,~e^*_\beta\in A^*, ~\alpha,~ \beta=1,2,\cdots,n$, by \eqref{cond-8}, we have
		\begin{eqnarray*}
			\mathrm{ad}^*_{e_i,e_j}(e^*_\alpha\circ e^*_\beta)=e^*_\alpha\circ(\mathrm{ad}^*_{e_i,e_j}e^*_\beta)+e^*_\beta\circ(\mathrm{ad}^*_{e_i,e_j}e^*_\alpha).
		\end{eqnarray*}
		Then
		\begin{eqnarray*}
			\mathrm{ad}^*_{e_i,e_j}(\sum_{l=1}^n{b_{\alpha\beta}^l e^*_l})=e^*_\alpha\circ(-\sum_{l=1}^n{c_{ijl}^\beta e_l^*})+e^*_\beta\circ(-\sum_{l=1}^n{c_{ijl}^\alpha e_l^*}).
		\end{eqnarray*}
		Furthermore
		\begin{eqnarray*}
			-\sum_{l=1}^n{b_{\alpha\beta}^l\sum_{m=1}^n{c_{ijm}^l e^*_m}}=-\sum_{l=1}^n{c_{ijl}^\beta\sum_{m=1}^n{b_{\alpha l}^m e^*_m} }-\sum_{l=1}^n{c_{ijl}^\alpha \sum_{m=1}^n{b_{\beta l}^m e^*_m}},
		\end{eqnarray*}
		which also gives \eqref{eq:sture-con}.	
		Thus, we deduce that \eqref{cond-1} is equivalent to \eqref{cond-8}.
		
		Secondly, using the same method as above, we can prove that \eqref{cond-2}$\Leftrightarrow$\eqref{cond-7}, \eqref{cond-3}$\Leftrightarrow$\eqref{cond-6}, and \eqref{cond-4}$\Leftrightarrow$\eqref{cond-5}. The conclusion holds.
	\end{proof}
	\begin{defi} 
		(\cite{Bai2019}) Let $(A,[\cdot,\cdot,\cdot])$ be a 3-Lie algebra and $\delta:A\rightarrow A\otimes A\otimes A$ be a linear map. If $(A,\delta)$ is a 3-Lie coalgebra and for all $x,y,z\in A$,
		\begin{align*}
			\delta([x,y,z])&=(1\otimes 1 \otimes \mathrm{ad}_{y,z})\delta(x) + (1\otimes 1 \otimes \mathrm{ad}_{z,x})\delta(y)+(1\otimes 1 \otimes \mathrm{ad}_{x,y})\delta(z),\\
			\delta([x,y,z]) &= (1\otimes 1 \otimes \mathrm{ad}_{y,z})\delta(x) + (1\otimes \mathrm{ad}_{y,z} \otimes 1)\delta(x)+(\mathrm{ad}_{y,z} \otimes 1\otimes 1 )\delta(x),
		\end{align*}
		then we call the triple $(A,[\cdot,\cdot,\cdot],\delta)$  a \textbf{double construction 3-Lie bialgebra}.
	\end{defi}
	\begin{defi} \label{defi:D-3-bi}
		Let $(A,\cdot, [\cdot,\cdot,\cdot])$ be a Poisson 3-Lie algebra and $(A,\Delta,\delta)$ be a Poisson 3-Lie coalgebra. If $(A,[\cdot,\cdot,\cdot],\delta)$ is a double construction 3-Lie bialgebra, $(A,\cdot,\Delta)$ is a commutative and cocommutative infinitesimal bialgebra satisfying:
		\begin{eqnarray}
			&&\Delta([x,y,z])=(1\otimes \mathrm{ad}_{y,z})\Delta(x)+(\mathrm{ad}_{y,z}\otimes 1 )\Delta(x), \label{bialg-1}\\
			&&\delta(x\cdot y)=\big(\mathcal{L}(y)\otimes1\otimes1\big)\delta(x)+\big(\mathcal{L}(x)\otimes1\otimes1\big)\delta(y), \label{bialg-2}\\
			&&\big(\mathcal{L}(x)\otimes1\otimes1\big)\delta(y)=\big(1\otimes1\otimes\mathcal{L}(x)\big)\delta(y), \label{bialg-3}\\
			&&(1\otimes\mathrm{ad}_{x,y})\Delta(z)=(1\otimes\mathrm{ad}_{y,z})\Delta(x), \quad\forall~ x,y,z\in A, \label{bialg-4}
		\end{eqnarray}
		then we call $(A, \cdot,[\cdot,\cdot,\cdot],\Delta,\delta)$ a \textbf{double construction Poisson 3-Lie bialgebra}.
	\end{defi}
	
	\begin{thm}\label{thm:one-one-co}
		Let $(A,\cdot,[\cdot,\cdot,\cdot])$ be a Poisson 3-Lie algebra and  $(A,\Delta,\delta)$ be a Poisson 3-Lie coalgebra. Then the following statements are equivalent: 
		
		$($$\mathrm{a}$$)$ $(A,\cdot, [\cdot,\cdot,\cdot],\Delta,\delta)$ is a double construction Poisson 3-Lie bialgebra;
		
		$($$\mathrm{b}$$)$ $(A,A^*,\mathrm{ad}^*,-\mathcal{L}^*,\mathfrak{ad}^*,-L^*)$ is a matched pair of Poisson 3-Lie algebras;
		
		$($$\mathrm{c}$$)$ $\big((A\oplus A^*,\mathcal{B}), A, A^*\big)$ is a standard Manin triple of Poisson 3-Lie algebras, where  $\mathcal{B}$,  $[\cdot,\cdot,\cdot]_{A\oplus A^*}$ and $\cdot_{A\oplus A^*}$ are given by \eqref{bilinear}, \eqref{Manin-1}, and \eqref{Manin-2}, respectively.
	\end{thm}
	\begin{proof}
		By Theorem \ref{Manin triple}, we have (b) is equivalent to (c). 
		By \cite[Theorem 4.14]{Bai2019}, we have $(A,[\cdot,\cdot,\cdot],\delta)$ is a double construction 3-Lie bialgebra if and only if $(A, A^*, \mathrm{ad}^*,\mathfrak{ad}^*)$ is a matched pair of 3-Lie algebras.  By \cite[Corollary 2.2.6]{Bai2010}, we have $(A,\cdot,\Delta)$ is a commutative and cocommutative infinitesimal bialgebra if and only if $(A,A^*,-\mathcal{L}^*,-L^*)$ is a matched pair of commutative associative algebras.
		Then, by Lemma \ref{equival},  to prove that (a) is equivalent to (b), we only need to prove that 
		$$\eqref{cond-1} \Leftrightarrow \eqref{bialg-1}, \quad\eqref{cond-2} \Leftrightarrow \eqref{bialg-2}, \quad\eqref{cond-3} \Leftrightarrow \eqref{bialg-3},  \quad\eqref{cond-4} \Leftrightarrow \eqref{bialg-4}.$$ 
		
		Let $\{e_1,e_2,\cdots,e_n\}$ be a basis of the vector space $A$. By \eqref{bialg-1}, for all $i,~j,~k=1,2,\cdots,n$, we have
		$$(1\otimes \mathrm{ad}_{e_j,e_k})\Delta(e_i)=(\mathrm{ad}_{e_k,e_j}\otimes 1 )\Delta(e_i)+\Delta([e_i,e_j,e_k]).$$
		Then using the structure constants in the proof of Lemma \ref{equival}, we have 
		$$(1\otimes \mathrm{ad}_{e_j,e_k})\sum_{\alpha,l=1}^n({b_{\alpha l}^i e_\alpha\otimes e_l})=(\mathrm{ad}_{e_k,e_j}\otimes 1)\sum_{m,l=1}^n({b_{l m}^i e_l\otimes e_m})+\Delta\big(\sum_{l=1}^n({c_{ijk}^l e_l})\big).$$
		Furthermore, for all $\alpha,~m,~l=1,2,\cdots,n$,
		$$\sum_{\alpha,l,m=1}^n({b_{\alpha l}^i c_{jkl}^m} e_\alpha\otimes e_m)=\sum_{\alpha,l,m=1}^n({b_{l m}^i c_{kjl}^\alpha} e_\alpha\otimes e_m)+\sum_{\alpha,m,l=1}^n({c_{ijk}^l b_{\alpha m}^l e_\alpha\otimes e_m}).$$	
		which  gives \eqref{eq:sture-con}. By the proof of Lemma \ref{equival},  we have \eqref{cond-1} $\Leftrightarrow$ \eqref{eq:sture-con}. Thus, we deduce that \eqref{cond-1} $\Leftrightarrow$ \eqref{bialg-1}.
		
		Similarly, using the same method as above, we can prove that  \eqref{cond-2} $\Leftrightarrow$ \eqref{bialg-2}, \eqref{cond-3} $\Leftrightarrow$ \eqref{bialg-3}, and \eqref{cond-4} $\Leftrightarrow$ \eqref{bialg-4}. Then (a) is equivalent to (b). The proof is completed.
	\end{proof}
	\section{Representations and matched pairs of transposed Poisson 3-Lie algebras}\label{sec:trans-Poisson}
	In this section, we introduce the notions of representation and matched pair for transposed Poisson 3-Lie algebras. As noted in \cite{G. Liu}, the method using Manin triples with invariant bilinear forms on both commutative associative algebras and Lie algebras does not apply to developing a bialgebra theory for transposed Poisson algebras. Furthermore, we prove that there is no such ``natural'' bialgebra theory for transposed Poisson 3-Lie algebras, as the mixed products of $\cdot$ and $[\cdot,\cdot,\cdot]$ are trivial (Proposition \ref{pro:no-bi-form}).

	\subsection{Representations of transposed Poisson 3-Lie algebras}\label{subsec:51}
	Recall \cite{G. Liu}, a representation of transposed Poisson algebra $(A, \cdot, [\cdot,\cdot])$ is a triple $(V, \rho, \mu)$, where $(V, \rho)$ and $(V, \mu)$ are respectively representations of Lie algebra $(A, [\cdot,\cdot])$ and commutative associative algebra $(A,\cdot)$ that satisfy
	\begin{align*}
		2\mu(x)\rho(y)&=\rho(x\cdot y)+\rho(y)\mu(x),\\
		2\mu([x,y])&=\rho(x)\mu(y)-\rho(y)\mu(x), \quad\forall~ x, y\in A.
	\end{align*}
	Replace the above representation of Lie algebra with the representation of 3-Lie algebra. By \eqref{eq:trans-poisson-alg}, we give the representation of the transposed Poisson 3-Lie algebra.
	\begin{defi}
		A \textbf{representation of  transposed Poisson 3-Lie algebra} $(B,\cdot, [\cdot,\cdot,\cdot])$ is a triple $(V,\rho,\mu)$, where $(V,\mu)$ is a representation of the commutative associative algebra $(B,\cdot)$, $(V,\rho)$ is a representation of the 3-Lie algebra $(B, [\cdot,\cdot,\cdot])$ such that
		\begin{align}
			3\mu(x)\rho(y,z)&=\rho(x\cdot y,z)+\rho(y,x\cdot z)+\rho(y,z)\mu(x),\label{trans:rep-1}\\
			3\mu([x,y,z])&=\rho(y,z)\mu(x)-\rho(x,z)\mu(y)+\rho(x,y)\mu(z),\quad \forall ~x,y,z\in B.\label{trans:rep-2}
		\end{align}	
	\end{defi}	
	Let $(B,\cdot, [\cdot,\cdot,\cdot])$ be a transposed Poisson 3-Lie algebra, $(B,\mathrm{ad})$ and $(B,\mathcal{L})$ be the adjoint representations of $(B, [\cdot,\cdot,\cdot])$ and $(B,\cdot)$, respectively. Then $(B,\mathrm{ad},\mathcal{L})$ is a representation of $(B,\cdot, [\cdot,\cdot,\cdot])$, called the \textbf{adjoint representation of transposed Poisson 3-Lie algebras}.	
	\begin{pro}\label{pro:trans-poisson-semi-direct}
		Let $(B,\cdot, [\cdot,\cdot,\cdot])$ be a transposed Poisson 3-Lie algebra and $V$ be a vector space. Let $\mu:B\rightarrow \mathfrak{gl}(V)$ and $\rho:\otimes^2 B\rightarrow\mathfrak{gl}(V)$ be linear maps. Then $(V,\rho,\mu)$ is a representation of $(B,\cdot, [\cdot,\cdot,\cdot])$ if and only if $(B\oplus V, \cdot_\mu, [\cdot,\cdot,\cdot]_\rho)$ is a  transposed Poisson 3-Lie algebra, where the product $\cdot_{\mu}$ and the bracket $[\cdot,\cdot,\cdot]_{\rho}$ are given by
		\begin{align}
			(x_1+u_1)\cdot_{\mu}(x_2+u_2)&=x_1\cdot x_2+ \mu(x_1)u_2+\mu(x_2)u_1,\label{eq:trans-poisson-semi-1}\\
			[x_1+u_1,x_2+u_2,x_3+u_3]_{\rho}&=[x_1,x_2,x_3]+\rho(x_1,x_2)u_3-\rho(x_1,x_3)u_2+\rho(x_2,x_3)u_1,\label{eq:trans-poisson-semi-2}
		\end{align}
		for all $x_1,x_2,x_3\in B, u_1,u_2,u_3\in V$. We denote it by $B\ltimes_{\mu,\rho}V$, which is called \textbf{semi-direct product transposed Poisson 3-Lie algebras}. 
	\end{pro}
	\begin{proof}
		By the proof of Proposition \ref{pro:possion-semi-direct}, we only need to prove that $(V,\rho,\mu)$ is a representation of $(B,\cdot, [\cdot,\cdot,\cdot])$ if and only if the product $\cdot_{\mu}$ and the bracket $[\cdot,\cdot,\cdot]_{\rho}$ satisfy \eqref{eq:trans-poisson-alg}.   If $(V,\rho,\mu)$ is a representation of $(B,\cdot, [\cdot,\cdot,\cdot])$, then for all $x_i \in B $, $u_i \in V,1\leq i\leq4$, we have
		\begin{eqnarray*}
			&&3(x_1+u_1)\cdot_\mu[x_2+u_2,x_3+u_3,x_4+u_4]_\rho \\
			&\overset{\eqref{eq:trans-poisson-semi-2}}{=}&3(x_1+u_1)\cdot_\mu([x_2,x_3,x_4]+\rho(x_2,x_3)u_4-\rho(x_2,x_4)u_3+\rho(x_3,x_4)u_2)\\
			&\overset{\eqref{eq:trans-poisson-semi-1}}{=}&3x_1\cdot[x_2,x_3,x_4]+3\mu(x_1)(\rho(x_2,x_3)u_4-\rho(x_2,x_4)u_3+\rho(x_3,x_4)u_2)+3\mu([x_2,x_3,x_4])u_1.\\
			&=&[x_1\cdot x_2,x_3,x_4]+[x_2,x_1\cdot x_3,x_4]+[x_2,x_3,x_1\cdot x_4]\\
			&\overset{\eqref{trans:rep-1},\eqref{trans:rep-2}}{=}&[x_1\cdot x_2,x_3,x_4]+\rho(x_1\cdot x_2,x_3)u_4-\rho(x_1\cdot x_2,x_4)u_3+\rho( x_3,x_4)(\mu(x_1)u_2+\mu(x_2)u_1)\\
			&&+[x_2,x_1\cdot x_3,x_4]+\rho(x_2,x_1\cdot x_3)u_4-\rho( x_2,x_4)\mu(x_1)(u_3+\mu(x_3)u_1)+\rho(x_1\cdot x_3,x_4)u_2\\
			&&+[x_2,x_3,x_1\cdot x_4]+\rho( x_2,x_3)(\mu(x_1)u_4+\mu(x_4)u_1)-\rho(x_2,x_1\cdot x_4)u_3+\rho(x_3,x_1\cdot x_4)u_2\\
			&\overset{\eqref{tp-semi-2}}{=}&[x_1\cdot x_2+\mu(x_1)u_2+\mu(x_2)u_1,x_3+u_3,x_4+u_4]_\rho \\
			&&+[x_2+u_2,x_1\cdot x_3+\mu(x_1)u_3+\mu(x_3)u_1,x_4+u_4]_\rho\\
			&&+[x_2+u_2,x_3+u_3,x_1\cdot x_4+\mu(x_1)u_4+\mu(x_4)u_1]_\rho\\
			&\overset{\eqref{tp-semi-1}}{=}&[(x_1+u_1)\cdot_\mu (x_2+u_2),x_3+u_3,x_4+u_4]_\rho+[x_2+u_2,(x_1+u_1)\cdot_\mu(x_3+u_3),x_4+u_4]_\rho\\
			&&+[x_2+u_2,x_3+u_3,(x_1+u_1)\cdot_\mu(x_4+u_4)]_\rho.
		\end{eqnarray*}
		
		Therefore $(B\oplus V, \cdot_\mu, [\cdot,\cdot,\cdot]_\rho)$ is a transposed Poisson 3-Lie algebra.  Conversely, in a similar manner, if $(B\oplus V, \cdot_\mu, [\cdot,\cdot,\cdot]_\rho)$ is a transposed Poisson 3-Lie algebra, then $(V,\rho,\mu)$ is a representation of Poisson 3-Lie algebras. The conclusion holds.
	\end{proof}
	\begin{pro}\label{pro:trans-poisson-dual-rep}
		Let $(B,\cdot, [\cdot,\cdot,\cdot])$ be a transposed Poisson 3-Lie algebra and $(V,\rho,\mu)$ be a representation of $(B,\cdot, [\cdot,\cdot,\cdot])$. Then $(V,\rho^*,-\mu^*)$ is a representation of $(B,\cdot, [\cdot,\cdot,\cdot])$ if and only if 
		\begin{equation}\label{dual-rep-prop}
			\mu(x)\rho(y,z)=\rho(y,z)\mu(x), \quad \mu([x,y,z])=0, \quad \forall~ x,y,z \in A.
		\end{equation}
		This representation of $(B,\cdot, [\cdot,\cdot,\cdot])$ is called the dual representation of $(V,\rho,\mu)$ .
		In particular,  $(B,ad^*,-\mathcal{L}^*)$ is a representation of $(B,\cdot, [\cdot,\cdot,\cdot])$ if and only if the following equation holds:
		\begin{equation}\label{eq:tran-poisson-dual}
			x\cdot [y,z,w]=[y,z,x\cdot w]=0, \quad \forall~ x,y,z,w \in B.
		\end{equation}	
	\end{pro}
	\begin{proof}
		For all $x,y,z\in B, \xi\in V^*, v\in V$, we have
		\begin{eqnarray*}
			&&\langle\big(-3\mu^*(x)\rho^*(y,z)-\rho^*(x\cdot y,z)-\rho^*(y,x\cdot z)+\rho^*(y,z)\mu^*(x)\big)\xi,v \rangle \nonumber \\ 
			&\overset{\eqref{com-asso-dual}\eqref{3-lie-dual}}{=}&\langle \xi, \big(-3\rho(y,z)\mu(x)+\rho(x\cdot y,z)+\rho(y,x\cdot z)+\mu(x)\rho(y,z)\big)v \rangle, \\
			&&\langle\big(-3\mu^*([x,y,z])+\rho^*(y,z)\mu^*(x)-\rho^*(x, z)\mu^*(y)+\rho^*(x,y)\mu^*(z)\big)\xi,v \rangle \nonumber \\
			&\overset{\eqref{com-asso-dual}\eqref{3-lie-dual}}{=}&\langle\xi,\big(3\mu([x,y,z])+\mu(x)\rho(y,z)-\mu(y)\rho(x, z)+\mu(z)\rho(x,y)\big)v \rangle.	
		\end{eqnarray*}
		Thus $(V,\rho^*,-\mu^*)$ is a representation if and only if the following equations hold:
		\begin{equation}\label{dual-rep-1}
			-3\rho(y,z)\mu(x)+\rho(x\cdot y,z)+\rho(y,x\cdot z)+\mu(x)\rho(y,z)=0,
		\end{equation}
		\begin{equation}\label{dual-rep-2}
			3\mu([x,y,z])+\mu(x)\rho(y,z)-\mu(y)\rho(x, z)+\mu(z)\rho(x,y)=0.
		\end{equation}
		By \eqref{trans:rep-1} and \eqref{trans:rep-2}, we have \eqref{dual-rep-1} and \eqref{dual-rep-2} hold if and only if \eqref{dual-rep-prop} holds. In particular, taking $\rho=\mathrm{ad}$ and $\mu=\mathcal{L}$, then $(B^*,\mathrm{ad}^*,-\mathcal{L}^*)$ is a representation of $(B,\cdot, [\cdot,\cdot,\cdot])$ if and only if \eqref{eq:tran-poisson-dual} holds.
	\end{proof}	
	\begin{rmk}
		If $(V,\rho,\mu)$ is a representation of a Poisson 3-Lie algebra $(A,\cdot, [\cdot,\cdot,\cdot])$, then $(V,\rho^*,-\mu^*)$ naturally be a representation of $(A,\cdot, [\cdot,\cdot,\cdot])$. However, there is no “natural” dual representation for transposed Poisson 3-Lie algebra. 
		%Example \ref{ex:poisson-and-tran-poisson} satisfies \eqref{eq:tran-poisson-dual}, and it is both a non-trivial Poisson 3-Lie algebra and a transposed Poisson 3-Lie algebra.
	\end{rmk}
	\subsection{Matched pairs of transposed Poisson 3-Lie algebras}\label{subsec:52}
	\begin{defi} \label{def:trans-poisson-matched-pair}
		Let $(A,\cdot_A,[\cdot,\cdot,\cdot]_A)$ and $(B,\cdot_B,[\cdot,\cdot,\cdot]_B)$ be two transposed Poisson 3-Lie algebras. Let $\mu_A:A\rightarrow \mathfrak{gl}(B)$, $\rho_A:\otimes^2 A\rightarrow \mathfrak{gl}(B)$ and $\mu_B:B\rightarrow \mathfrak{gl}(A)$, $\rho_B:\otimes^2 B\rightarrow \mathfrak{gl}(A)$ be linear maps such that $(A,B,\mu_A,\mu_B)$ is a matched pair of commutative associative algebra and $(A,B,\rho_A,\rho_B)$ is a matched pair of 3-Lie algebra. If $(B,\rho_A,\mu_A)$ is a representation of $(A,\cdot_A,[\cdot,\cdot,\cdot]_A)$ and $(A,\rho_B,\mu_B)$ is a representation of $(B,\cdot_B,[\cdot,\cdot,\cdot]_B)$ and for all $x_i\in A, y_i\in B, 1\leq i\leq 4$, the following equations hold:
		%	除了\small 之外，大小依次递减还可以选择\footnotesize，\scriptsize
		{\footnotesize \begin{eqnarray}
				&&3\mu_B(y_4)[x_1,x_2,x_3]_A=[\mu_B(y_4)x_1,x_2,x_3]_A+[x_1,\mu_B(y_4)x_2,x_3]_A
				+[x_1,x_2,\mu_B(y_4)x_3]_A,\label{match-cond-1}\\ 
				&&3\mu_A(x_4)[y_1,y_2,y_3]_B=[\mu_A(x_4)y_1,y_2,y_3]_B+[y_1,\mu_A(x_4)y_2,y_3]_B+[y_1,y_2,\mu_A(x_4)y_3]_B,\label{match-cond-2}\\ 
				&&[\mu_B(y_1)x_4,x_2,x_3]_A=3\mu_B(\rho_A(x_2,x_3)y_1)x_4, \label{match-cond-3}\\ 
				&&[\mu_A(x_1)y_4,y_2,y_3]_B=3\mu_A(\rho_B(y_2,y_3)x_1)y_4, \label{match-cond-4}\\ 
				&&\rho_B(y_1,y_2)(x_4\cdot_A x_3)+\rho_B(\mu_A(x_4)y_1,y_2)x_3+\rho_B(y_1,\mu_A(x_4)y_2)x_3+\rho_A(\mu_B(y_2)x_4,x_3)y_1\nonumber \\
				&&+\rho_A(x_3, \mu_B(y_1)x_4)y_2=0,\label{match-cond-5}\\ 
				&&\rho_A(x_1,x_2)(y_4\cdot_B y_3)+\rho_A(\mu_B(y_4)x_1,x_2)y_3+\rho_A(x_1,\mu_B(y_4)x_2)y_3+\rho_B(\mu_A(x_2)y_4,y_3)x_1\nonumber \\
				&&+\rho_B(y_3,\mu_A(x_1)y_4)x_2=0,	\label{match-cond-6}
		\end{eqnarray}}
		then $(A,B,\rho_A,\mu_A,\rho_B,\mu_B)$ is called a \textbf{matched pair of transposed Poisson 3-Lie algebras}.
	\end{defi}
	\begin{thm}
		Let $(A,\cdot_A,[\cdot,\cdot,\cdot]_A)$ and $(B,\cdot_B,[\cdot,\cdot,\cdot]_B)$ be two transposed Poisson 3-Lie algebras. Suppose that $\mu_A:A\rightarrow \mathfrak{gl}(B)$, $\rho_A:\otimes^2 A\rightarrow \mathfrak{gl}(B)$ and $\mu_B:B\rightarrow \mathfrak{gl}(A)$, $\rho_B:\otimes^2 B\rightarrow \mathfrak{gl}(A)$ are linear maps. Define two operations $\cdot$ and $[\cdot,\cdot,\cdot]$ on $A\oplus B$ by \eqref{com-asso-sum} and \eqref{3-lie-sum} respectively. Then $(A\oplus B,\cdot,[\cdot,\cdot,\cdot])$ is a transposed Poisson 3-Lie algebra if and only if $(A,B,\rho_A,\mu_A,\rho_B,\mu_B)$ is a matched pair of transposed Poisson 3-Lie algebras. 
		%We denote it  by $A\Join _{\mu_A,\rho_A}^{\mu_B,\rho_B}B$.
	\end{thm}
	\begin{proof} 
		By Proposition \ref{pro:direct-3-Lie} and Proposition \ref{pro:direct-comm-asso},  $(A\oplus B,\cdot,[\cdot,\cdot,\cdot])$ is a transposed Poisson 3-Lie algebra if and only if for all $x_i\in A, y_i\in B, 1\leq i\leq 4$,
		\begin{eqnarray*}
			&&\underbrace{3(x_4+y_4)\cdot[x_1+y_1,x_2+y_2,x_3+y_3]}_{(B1)}=\underbrace{[(x_4+y_4)\cdot(x_1+y_1),x_2+y_2,x_3+y_3]}_{(B2)}\\
			&&+\underbrace{[x_1+y_1,(x_4+y_4)\cdot(x_2+y_2),x_3+y_3]}_{(B3)}+\underbrace{[x_1+y_1,x_2+y_2,(x_4+y_4)\cdot(x_3+y_3)]}_{(B4)}.
		\end{eqnarray*}	
		By \eqref{com-asso-sum} and \eqref{3-lie-sum}, we have
		\begin{eqnarray*}
			(B1)&=&3(x_4+y_4)\cdot\big([x_1,x_2,x_3]_A+\rho_B(y_1,y_2)x_3+\rho_B(y_3,y_1)x_2+\rho_B(y_2,y_3)x_1  \\
			&&+[y_1,y_2,y_3]_B+\rho_A(x_1,x_2)y_3+\rho_A(x_3,x_1)y_2+\rho_A(x_2,x_3)y_1\big)\\
			&=&\underbrace{3x_4\cdot_A[x_1,x_2,x_3]_A}_{(a1)}+\underbrace{3y_4\cdot_B[y_1,y_2,y_3]_B}_{(a5)}+\underbrace{3\mu_B(y_4)[x_1,x_2,x_3]_A}_{(d1)}\\
			&&+\underbrace{3\mu_B(y_4)\big(\rho_B(y_1,y_2)x_3\big)}_{(b1)}+\underbrace{3\mu_B(y_4)\big(\rho_B(y_3,y_1)x_2\big)}_{(b5)}+\underbrace{3\mu_B(y_4)\big(\rho_B(y_2,y_3)x_1\big)}_{(b9)}\\
			&&+\underbrace{3\mu_B([y_1,y_2,y_3]_B)x_4}_{(c1)}+\underbrace{3\mu_B\big(\rho_A(x_1,x_2)y_3\big)x_4}_{(f6)}+\underbrace{3\mu_B\big(\rho_A(x_3,x_1)y_2\big)x_4}_{(f4)}\\
			&&+\underbrace{3\mu_B\big(\rho_A(x_2,x_3)y_1\big)x_4}_{(f2)}+\underbrace{3\mu_A(x_4)[y_1,y_2,y_3]_B}_{(e1)}+\underbrace{3\mu_A(x_4)\big(\rho_A(x_1,x_2)y_3\big)}_{(b13)}\\
			&&+\underbrace{3\mu_A(x_4)\big(\rho_A(x_3,x_1)y_2\big)}_{(b17)}+\underbrace{3\mu_A(x_4)\big(\rho_A(x_2,x_3)y_1\big)}_{(b21)}+\underbrace{3\mu_A([x_1,x_2,x_3]_A)y_4}_{(c5)}\\
			&&+\underbrace{3\mu_A\big(\rho_B(y_1,y_2)x_3\big)y_4}_{(g6)}+\underbrace{3\mu_A\big(\rho_B(y_3,y_1)x_2\big)y_4}_{(g4)}+\underbrace{3\mu_A\big(\rho_B(y_2,y_3)x_1\big)y_4}_{(g2)},\\
			(B2)&=&[x_4\cdot_A x_1 + \mu_B(y_4)x_1+ \mu_B(y_1)x_4+y_4\cdot_B y_1+\mu_A(x_4)y_1+\mu_A(x_1)y_4,x_2+y_2,x_3+y_3]\\
			&=&\underbrace{[x_4\cdot_A x_1,x_2,x_3]_A}_{(a2)}+\underbrace{[\mu_B(y_4)x_1,x_2,x_3]_A}_{(d2)}+\underbrace{[\mu_B(y_1)x_4,x_2,x_3]_A}_{(f1)}\\
			&&+\underbrace{\rho_B(y_4\cdot_B y_1,y_2)x_3}_{(b2)}+\underbrace{\rho_B\big(\mu_A(x_4)y_1,y_2\big)x_3}_{(h12)}+\underbrace{\rho_B\big(\mu_A(x_1)y_4,y_2\big)x_3}_{(j7)}  \\
			&&+\underbrace{\rho_B(y_3,y_4\cdot_B y_1)x_2}_{(b6)}+\underbrace{\rho_B\big(y_3,\mu_A(x_4)y_1\big)x_2}_{(h7)}+\underbrace{\rho_B\big(y_3,\mu_A(x_1)y_4\big)x_2}_{(j12)}\\
			&&+\underbrace{\rho_B(y_2,y_3)(x_4\cdot_A x_1)}_{(h1)} + \underbrace{\rho_B(y_2,y_3)\mu_B(y_4)x_1}_{(b10)}+ \underbrace{\rho_B(y_2,y_3)\mu_B(y_1)x_4}_{(c2)}\\
			&&+\underbrace{[y_4\cdot y_1,y_2,y_3]_B}_{(a6)}+\underbrace{[\mu_A(x_4)y_1,y_2,y_3]_B}_{(e2)}+\underbrace{[\mu_A(x_1)y_4,y_2,y_3]_B}_{(g1)}\\
			&&+\underbrace{\rho_A(x_4\cdot_A x_1,x_2)y_3}_{(b14)}+\underbrace{\rho_A\big(\mu_B(y_4)x_1,x_2\big)y_3}_{(j13)}+\underbrace{\rho_A\big(\mu_B(y_1)x_4,x_2\big)y_3}_{(h8)}\\
			&&+\underbrace{\rho_A(x_3,x_4\cdot_A x_1)y_2}_{(b18)}+\underbrace{\rho_A\big(x_3,\mu_B(y_4)x_1\big)y_2}_{(j8)}+\underbrace{\rho_A\big(x_3,\mu_B(y_1)x_4\big)y_2}_{(h13)}\\
			&&+\underbrace{\rho_A(x_2,x_3)(y_4\cdot_B y_1)}_{(j1)}+\underbrace{\rho_A(x_2,x_3)\mu_A(x_4)y_1}_{(b22)}+\underbrace{\rho_A(x_2,x_3)\mu_A(x_1)y_4}_{(c6)},\\
			(B3)&=&[x_1+y_1,x_4\cdot_A x_2 + \mu_B(y_4)x_2+ \mu_B(y_2)x_4+y_4\cdot_B y_2+\mu_A(x_4)y_2+\mu_A(x_2)y_4,x_3+y_3]\\
			&=&\underbrace{[x_1,x_4\cdot_A x_2,x_3]_A}_{(a3)}+\underbrace{[x_1, \mu_B(y_4)x_2,x_3]_A}_{(d3)}+\underbrace{[x_1,\mu_B(y_2)x_4,x_3]_A}_{(f3)}\\
			&&+\underbrace{\rho_B(y_1,y_4\cdot_B y_2)x_3}_{(b3)} +\underbrace{\rho_B\big(y_1,\mu_A(x_4)y_2\big)x_3}_{(h14)}+\underbrace{\rho_B\big(y_1,\mu_A(x_2)y_4\big)x_3}_{(j2)} \\
			&&+\underbrace{\rho_B(y_3,y_1)(x_4\cdot_A x_2)}_{(h6)} + \underbrace{\rho_B(y_3,y_1)\mu_B(y_4)x_2}_{(b7)}+ \underbrace{\rho_B(y_3,y_1)\mu_B(y_2)x_4}_{(c3)}\\
			&&+\underbrace{\rho_B(y_4\cdot_B y_2,y_3)x_1}_{(b11)}+\underbrace{\rho_B\big(\mu_A(x_4)y_2,y_3\big)x_1}_{(h2)}+\underbrace{\rho_B\big(\mu_A(x_2)y_4,y_3\big)x_1}_{(j14)}\\
			&&+\underbrace{[y_1,y_4\cdot_B y_2,y_3]_B}_{(a7)}+\underbrace{[y_1,\mu_A(x_4)y_2,y_3]_B}_{(e3)}+\underbrace{[y_1,\mu_A(x_2)y_4,y_3]_B}_{(g3)}\\
			&&+\underbrace{\rho_A(x_1,x_4\cdot_A x_2)y_3}_{(b15)}+\underbrace{\rho_A\big(x_1,\mu_B(y_4)x_2\big)y_3}_{(j15)}+\underbrace{\rho_A\big(x_1,\mu_B(y_2)x_4\big)y_3}_{(h3)}\\
			&&+\underbrace{\rho_A(x_3,x_1)(y_4\cdot_B y_2)}_{(j6)}+\underbrace{\rho_A(x_3,x_1)\mu_A(x_4)y_2}_{(b19)}+\underbrace{\rho_A(x_3,x_1)\mu_A(x_2)y_4}_{(c7)}\\
			&&+\underbrace{\rho_A(x_4\cdot_A x_2,x_3)y_1}_{(b23)}+\underbrace{\rho_A\big(\mu_B(y_4)x_2,x_3\big)y_1}_{(j3)}+\underbrace{\rho_A\big(\mu_B(y_2)x_4,x_3\big)y_1}_{(h15)},\\
			(B4)&=&[x_1+y_1,x_2+y_2,x_4\cdot_A x_3 + \mu_B(y_4)x_3+ \mu_B(y_3)x_4+y_4\cdot_B y_3+\mu_A(x_4)y_3+\mu_A(x_3)y_4]\\
			&=&\underbrace{[x_1,x_2,x_4\cdot_A x_3]_A}_{(a4)}+\underbrace{[x_1,x_2,\mu_B(y_4)x_3]_A}_{(d4)}+\underbrace{[x_1,x_2,\mu_B(y_3)x_4]_A}_{(f5)}\\
			&&+\underbrace{\rho_B(y_1,y_2)(x_4\cdot_A x_3)}_{(h11)} + \underbrace{\rho_B(y_1,y_2)\mu_B(y_4)x_3}_{(b4)}+ \underbrace{\rho_B(y_1,y_2)\mu_B(y_3)x_4}_{(c4)}  \\
			&&+\underbrace{\rho_B(y_4\cdot_B y_3,y_1)x_2}_{(b8)}+\underbrace{\rho_B\big(\mu_A(x_4)y_3,y_1\big)x_2}_{(h9)}+\underbrace{\rho_B\big(\mu_A(x_3)y_4,y_1\big)x_2}_{(j4)}\\
			&&+\underbrace{\rho_B(y_2,y_4\cdot_B y_3)x_1}_{(b12)}+\underbrace{\rho_B\big(y_2,\mu_A(x_4)y_3\big)x_1}_{(h4)}+\underbrace{\rho_B\big(y_2,\mu_A(x_3)y_4\big)x_1}_{(j9)}\\
			&&+\underbrace{[y_1,y_2,y_4\cdot_B y_3]_B}_{(a8)}+\underbrace{[y_1,y_2,\mu_A(x_4)y_3]_B}_{(e4)}+\underbrace{[y_1,y_2,\mu_A(x_3)y_4]_B}_{(g5)}\\
			&&+\underbrace{\rho_A(x_1,x_2)(y_4\cdot_B y_3)}_{(j11)}+\underbrace{\rho_A(x_1,x_2)\mu_A(x_4)y_3}_{(b16)}+\underbrace{\rho_A(x_1,x_2)\mu_A(x_3)y_4}_{(c8)}\\
			&&+\underbrace{\rho_A(x_4\cdot_A x_3,x_1)y_2}_{(b20)}+\underbrace{\rho_A\big(\mu_B(y_4)x_3,x_1\big)y_2}_{(j10)}+\underbrace{\rho_A\big( \mu_B(y_3)x_4,x_1\big)y_2}_{(h5)}\\
			&&+\underbrace{\rho_A(x_2,x_4\cdot_A x_3)y_1}_{(b24)}+\underbrace{\rho_A\big(x_2,\mu_B(y_4)x_3\big)y_1}_{(j5)}+\underbrace{\rho_A\big(x_2, \mu_B(y_3)x_4\big)y_1}_{(h10)}.	
		\end{eqnarray*}
		By \eqref{eq:trans-poisson-alg}, we have $$(a1)=(a2)+(a3)+(a4),~ (a5)=(a6)+(a7)+(a8).$$
		By \eqref{trans:rep-1}, we have $$(b1)=(b2)+(b3)+(b4), ~(b5)=(b6)+(b7)+(b8), ~(b9)=(b10)+(b11)+(b12),$$ 
		$$(b13)=(b14)+(b15)+(b16), ~(b17)=(b18)+(b19)+(b20), ~(b21)=(b22)+(b23)+(b24).$$
		By \eqref{trans:rep-2}, we have $$(c1)=(c2)+(c3)+(c4), ~(c5)=(c6)+(c7)+(c8).$$
		Moreover, we have the following equivalences
		\begin{eqnarray*}
			(d1)=(d2)+(d3)+(d4) &\Leftrightarrow& \eqref{match-cond-1};\\
			(e1)=(e2)+(e3)+(e4) &\Leftrightarrow& \eqref{match-cond-2};\\
			(f1)=(f2), ~(f3)=(f4), ~(f5)=(f6) 	&\Leftrightarrow& \eqref{match-cond-3};\\
			(g1)=(g2), ~(g3)=(g4), ~(g5)=(g6) &\Leftrightarrow& \eqref{match-cond-4};\\
			(h1)+(h2)+(h3)+(h4)+(h5)=0 &\Leftrightarrow& \eqref{match-cond-5};\\ 
			(h6)+(h7)+(h8)+(h9)+(h10)=0 &\Leftrightarrow& \eqref{match-cond-5};\\
			(h11)+(h12)+(h13)+(h14)+(h15)=0&\Leftrightarrow& \eqref{match-cond-5};\\
			(j1)+(j2)+(j3)+(j4)+(j5)=0 &\Leftrightarrow& \eqref{match-cond-6};\\
			(j6)+(j7)+(j8)+(j9)+(j10)=0 &\Leftrightarrow& \eqref{match-cond-6};\\
			(j11)+(j12)+(j13)+(j14)+(j15)=0 &\Leftrightarrow& \eqref{match-cond-6}.
		\end{eqnarray*}
		Therefore, $(A,B,\rho_A,\mu_A,\rho_B,\mu_B)$ is a matched pair of transposed Poisson 3-Lie algebras if and only if $(A\oplus B,\cdot,[\cdot,\cdot,\cdot])$ is  a transposed  Poisson 3-Lie algebra. 	
	\end{proof}
	\begin{pro} \label{pro:no-bi-form}
		Let $(B,\cdot,[\cdot,\cdot,\cdot])$ be a transposed Poisson 3-Lie algebra. If there is a nondegenerate invariant symmetric bilinear form $\mathcal{B}$ such that it is invariant on both $(B,\cdot)$ and $(B,[\cdot,\cdot,\cdot])$, i.e., $\mathcal{B}(x\cdot y,z)=\mathcal{B}(x, y\cdot z)$ and $\mathcal{B}([x,y,z],w)=-\mathcal{B}([x,y,w],z)$, for all $ x,y,z,w\in B$, then we have 
		\begin{equation}\label{tp-alg-bilinear-form}
			w\cdot[x,y,z]=[w\cdot x,y,z]=0.
		\end{equation}
	\end{pro}
	\begin{proof}
		For all $x,y,z,w,u\in B$, by the invariance on both $(B,\cdot)$ and $(B,[\cdot,\cdot,\cdot])$, we have 
		\begin{eqnarray*}
			0&=&\mathcal{B}(3u\cdot[x,y,z]-[u\cdot x,y,z]-[x,u\cdot y,z]-[x,y,u\cdot z],w)\\
			&=&\mathcal{B}(3w\cdot[x,y,z]+[y,z,w]\cdot x-[x,z,w]\cdot y+[x,y,w]\cdot z,u)\\
			&=&\mathcal{B}(-[x,y,3u\cdot w]+[u\cdot x,y,w]+[x,u\cdot y,w]\cdot y+[x,y,w]\cdot u,w).
		\end{eqnarray*}
		By the nondegeneracy of $\mathcal{B}$, we obtain
		\begin{eqnarray}\label{eq:mix1}
			3w\cdot[x,y,z]+[y,z,w]\cdot x-[x,z,w]\cdot y+[x,y,w]\cdot z=0,
\end{eqnarray}
and
\begin{eqnarray}\label{eq:mix2}
			-[x,y,3u\cdot w]+[u\cdot x,y,w]+[x,u\cdot y,w]\cdot y+[x,y,w]\cdot u=0.
		\end{eqnarray}
		Applying \eqref{eq:trans-poisson-alg} to \eqref{eq:mix1}, we obtain
		\begin{eqnarray*}
			&&3w\cdot[x,y,z]+[y,z,w]\cdot x-[x,z,w]\cdot y+[x,y,w]\cdot z\\
			&\overset{\eqref{eq:trans-poisson-alg}}{=}&[w\cdot x,y,z]+[x,w\cdot y,z]+[x,y,w\cdot z]+\frac{1}{3}([x\cdot y,z,w]+[y,x\cdot z,w]+[y,z,x\cdot w]\\
			&&-[y\cdot x,z,w]-[x,y\cdot z,w]-[x,z,y\cdot w]+[z\cdot x,y,w]+[x,z\cdot y,w]+[x,y,z\cdot w])\\
			&=&\frac{4}{3}([w\cdot x,y,z]+[x,w\cdot y,z]+[x,y,w\cdot z])\\
			&\overset{\eqref{eq:trans-poisson-alg}}{=}&\frac{4}{9}(w\cdot[x,y,z])\\
			&=&0.
		\end{eqnarray*}
Therefore, 	$w\cdot[x,y,z]=0$.	Similarly, applying \eqref{eq:trans-poisson-alg} to \eqref{eq:mix2}, we obtain $[w\cdot x,y,z]=0$. The conclusion holds.
	\end{proof}
	\begin{rmk}
		If the mixed products  $w\cdot[x,y,z]\neq0$ or $[w\cdot x,y,z]\neq0$ in a transposed Poisson 3-Lie algebra $(B,\cdot,[\cdot,\cdot,\cdot])$, then there is no non-trivial, non-degenerate invariant symmetric bilinear form $\mathcal{B}$ on $(B,\cdot,[\cdot,\cdot,\cdot])$. Proposition \ref{pro:no-bi-form} shows that the existence of such a non-trivial form would make the transposed 3-Lie Leibniz rule \eqref{eq:trans-poisson-alg} trivial, and thus there is no ``natural'' bialgebra theory for transposed Poisson 3-Lie algebras as there is for Poisson algebras in \cite{X. Ni}.
	\end{rmk}
	\section{Double construction admissible transposed  Poisson 3-Lie bialgebras}\label{sec:ad-poisson}
	
	Although there is no ``natural'' bialgebra theory for transposed Poisson 3-Lie algebras, we can still use the double construction Poisson 3-Lie bialgebras from Section \ref{sec:double-Poisson} to induce bialgebras for transposed Poisson  3-Lie algebras. This is because the intersection of transposed 3-Lie Poisson algebras and Poisson 3-Lie algebras includes admissible transposed Poisson 3-Lie algebras. This approach differs from Liu and Bai's method, which uses a commutative 2-cocycle on a Lie algebra to present the bialgebra theory of transposed Poisson algebras \cite{G. Liu}.
	
	In this section, we introduce representations, matched pairs, and double construction admissible transposed Poisson 3-Lie bialgebras. We prove that a matched pair of admissible transposed Poisson 3-Lie algebras is equivalent to a double construction admissible transposed Poisson 3-Lie bialgebra (Theorem \ref{thm:adm-m-d}). Additionally, we provide a non-trivial example of such a bialgebra (Example \ref{ex:no-bi}).
	
	\subsection{Representations of admissible transposed Poisson 3-Lie algebras}
	\begin{pro}
		Let $(V, \mu)$  be a  representation  of commutative associative algebra $(A,\cdot)$ and $(V, \rho)$ be a  representation  of  3-Lie algebra $(A,[\cdot,\cdot,\cdot])$. If $(V,\rho, \mu)$ is both  a  representation  of Poisson 3-Lie algebra $(A,\cdot,[\cdot,\cdot,\cdot])$ and  a  representation  of transposed Poisson 3-Lie algebra $(A,\cdot,[\cdot,\cdot,\cdot])$ if and only if	
		\begin{equation}\label{eq:trans-poisson-and-poisson-rep-prop}
			\mu(x)\rho(y,z)=\rho(y,z)\mu(x)=\rho(x\cdot y,z)=\mu([x,y,z])=0, \quad \forall~ x,y,z \in A.
		\end{equation}
	\end{pro}
	\begin{proof}
		For all $x,y,z\in A$, by \eqref{rep-2} and \eqref{trans:rep-2}, we have
		\begin{align*}
			3\mu([x,y,z])&=\mu([y,z,x])+\mu(x)\rho(y,z)-\mu([x,z,y])-\mu(y)\rho(x,z)+\mu([x,y,z])+\mu(z)\rho(x,y)\\
			&=3\mu([x,y,z])+\mu(x)\rho(y,z)+\mu(y)\rho(z,x)+\mu(z)\rho(x,y).
		\end{align*}
		Then
		\begin{align*}
			\mu(x)\rho(y,z)+\mu(y)\rho(z,x)+\mu(z)\rho(x,y)=0.
		\end{align*}
		By \eqref{trans:rep-1}, we have 
		\begin{align*}
			0&=3\mu(x)\rho(y,z)+3\mu(y)\rho(z,x)+3\mu(z)\rho(x,y)\\
			&=\rho(x\cdot y,z)+\rho(y,x\cdot z)+\rho(y,z)\mu(x)+\rho(y\cdot z,x)+\rho(z,x\cdot y)+\rho(z,x)\mu(y)\\
			&+\rho(z\cdot x, y)+\rho(x,y\cdot z)+\rho(x,y)\mu(z).
		\end{align*}
		Then
		\begin{align*}
			0=\rho(y,z)\mu(x)+\rho(z,x)\mu(y)+\rho(x,y)\mu(z)=3\mu([x,y,z]).
		\end{align*}
		Furthermore, by \eqref{rep-2} and \eqref{trans:rep-1}, we have 
		\begin{align*}
			\mu(x)\rho(y,z)&=\rho(y,z)\mu(x),\\
			2\mu(x)\rho(y,z)&=\rho(x\cdot y,z)+\rho(y,x\cdot z).
		\end{align*} 
		Moreover, by \eqref{rep-1} and \eqref{trans:rep-2}, we have
		\begin{align*}
			0&=\rho(y,z)\mu(x)+\rho(z,x)\mu(y)+\rho(x,y)\mu(z)\\
			&=\rho(y,z)\mu(x)+\rho(x,z)\mu(y)+2\rho(z,x)\mu(y)+\rho(x,y)\mu(z)\\
			&=\rho(x\cdot y, z)+\rho(y\cdot z,x)+\rho(z,y\cdot x)+\rho(x,y)\mu(z)\\
			&=\rho(x,y)\mu(z)+\rho(y\cdot z,x).
		\end{align*}
		Then $$\mu(z)\rho(y,x)=\rho(y\cdot z,x).$$
		By \eqref{rep-1}, we have 
		$$0=\mu(x)\rho(y,z)+\mu(y)\rho(x,z)-\rho(x\cdot y,z)=\mu(x)\rho(y,z).$$
		Therefore, \eqref{eq:trans-poisson-and-poisson-rep-prop} holds, and the converse also holds.  The proof is completed.
	\end{proof}
	\begin{defi}
		Let $(V,\mu)$ be a  representation of commutative associative algebra $(B,\cdot)$ and $(V,\rho)$ be  a representation of 3-Lie algebra $(B,[\cdot,\cdot,\cdot])$. If $\mu$ and $\rho$ satisfy \eqref{eq:trans-poisson-and-poisson-rep-prop}, then the triple $(V,\rho,\mu)$ is called a \textbf{representation of an admissible transposed Poisson 3-Lie algebra}.
	\end{defi}
	Let $(A,\cdot, [\cdot,\cdot,\cdot])$ be an admissible transposed Poisson 3-Lie algebra, $(A,\mathrm{ad})$ and $(A,\mathcal{L})$ be the adjoint representations of $(A, [\cdot,\cdot,\cdot])$ and $(A,\cdot)$, respectively. Then $(A,\mathrm{ad},\mathcal{L})$ is a representation of $(A,\cdot, [\cdot,\cdot,\cdot])$, which is called the \textbf{adjoint representation of  admissible transposed Poisson 3-Lie algebras}.
	\begin{pro}
		Let $(A,\cdot, [\cdot,\cdot,\cdot])$ be an admissible transposed Poisson 3-Lie algebra and $(V,\rho,\mu)$ be a representation of $(A,\cdot, [\cdot,\cdot,\cdot])$. Then $(V^*,\rho^*,-\mu^*)$ is a representation of $(A,\cdot, [\cdot,\cdot,\cdot])$. This representation is called the dual representation of $(V,\rho,\mu)$. In particular,  $(A^*,\mathrm{ad}^*,-\mathcal{L}^*)$ is a representation of $(A,\cdot, [\cdot,\cdot,\cdot])$, which is called the \textbf{coadjoint representation of admissible transposed  Poisson 3-Lie algebras}. 
	\end{pro}
	\begin{proof}
		It is follows from Proposition \ref{pro:poisson-dual} and Proposition \ref{pro:trans-poisson-dual-rep}.
	\end{proof}
	\begin{pro}
		Let $(B,\cdot, [\cdot,\cdot,\cdot])$ be an admissible transposed Poisson 3-Lie algebra and $V$ be a vector space. Let $\mu:B\rightarrow \mathfrak{gl}(V)$ and $\rho:\otimes^2 B\rightarrow\mathfrak{gl}(V)$ be linear maps. Then $(V,\rho,\mu)$ is a representation of $(B,\cdot, [\cdot,\cdot,\cdot])$ if and only if $(B\oplus V, \cdot_\mu, [\cdot,\cdot,\cdot]_\rho)$ is an admissible  transposed Poisson 3-Lie algebra, where the product $\cdot_{\mu}$ and the bracket $[\cdot,\cdot,\cdot]_{\rho}$ are given by
		\begin{align*}
			(x_1+u_1)\cdot_{\mu}(x_2+u_2)&=x_1\cdot x_2+ \mu(x_1)u_2+\mu(x_2)u_1,\\
			[x_1+u_1,x_2+u_2,x_3+u_3]_{\rho}&=[x_1,x_2,x_3]+\rho(x_1,x_2)u_3-\rho(x_1,x_3)u_2+\rho(x_2,x_3)u_1,
		\end{align*}
		for all $x_1,x_2,x_3\in B, u_1,u_2,u_3\in V$. We denote it by $B\ltimes_{\mu,\rho}V$, which is called \textbf{semi-direct product admissible transposed Poisson 3-Lie algebras}. 
	\end{pro}
	\begin{proof}
		It follows from Proposition \ref{pro:possion-semi-direct} and Proposition \ref{pro:trans-poisson-semi-direct}.
	\end{proof}
	\subsection{Matched pairs of admissible transposed Poisson 3-Lie algebras}
	\begin{defi} \label{def:admiss-trans-poisson-matched-pair}
		Let $(A,\cdot_A,[\cdot,\cdot,\cdot]_A)$ and $(B,\cdot_B,[\cdot,\cdot,\cdot]_B)$ be two admissible transposed Poisson 3-Lie algebras. Let $\mu_A:A\rightarrow \mathfrak{gl}(B)$, $\rho_A:\otimes^2 A\rightarrow \mathfrak{gl}(B)$ and $\mu_B:B\rightarrow \mathfrak{gl}(A)$, $\rho_B:\otimes^2 B\rightarrow \mathfrak{gl}(A)$ be linear maps such that $(A,B,\mu_A,\mu_B)$ is a matched pair of commutative associative algebra and $(A,B,\rho_A,\rho_B)$ is a matched pair of 3-Lie algebra. If $(B,\rho_A,\mu_A)$ is a representation of $(A,\cdot_A,[\cdot,\cdot,\cdot]_A)$ and $(A,\rho_B,\mu_B)$ is a representation of $(B,\cdot_B,[\cdot,\cdot,\cdot]_B)$ and for all $x_i\in A, y_i\in B, 1\leq i\leq 4$, the following equations hold:
		\begin{eqnarray}
			&&\mu_B(y_4)[x_1,x_2,x_3]_A=\rho_A(x_1,x_2)(y_4\cdot_B y_3)=0,\label{eq:admiss-match-pair-1}\\
			&&\mu_A(x_4)[y_1,y_2,y_3]_B=\rho_B(y_1,y_2)(x_4\cdot_A x_3)=0,\label{eq:admiss-match-pair-2}\\ 
			&&\mu_B\big(\rho_A(x_2,x_3)y_1\big)x_4=0,\label{eq:admiss-match-pair-3}\\ 
			&&\mu_A\big(\rho_B(y_2,y_3)x_1\big)y_4=0,\label{eq:admiss-match-pair-4}\\ 
			&&[\mu_B(y_1)x_4,x_2,x_3]_A=0,\label{eq:admiss-match-pair-5}\\ 
			&&[\mu_A(x_1)y_4,y_2,y_3]_B=0,\label{eq:admiss-match-pair-6}\\ 
			&&\rho_A\big(\mu_B(y_4)x_1,x_2\big)y_3+\rho_B\big(y_3,\mu_A(x_1)y_4\big)x_2=0,\label{eq:admiss-match-pair-7}	
		\end{eqnarray}
		then $(A,B,\rho_A,\mu_A,\rho_B,\mu_B)$ is called a \textbf{matched pair of admissible transposed Poisson 3-Lie algebras}.
	\end{defi}
	
	\begin{thm}
		Let $(A,\cdot_A,[\cdot,\cdot,\cdot]_A)$ and $(B,\cdot_B,[\cdot,\cdot,\cdot]_B)$ be two admissible transposed Poisson 3-Lie algebras. Suppose that $\mu_A:A\rightarrow \mathfrak{gl}(B)$, $\rho_A:\otimes^2 A\rightarrow \mathfrak{gl}(B)$ and $\mu_B:B\rightarrow \mathfrak{gl}(A)$, $\rho_B:\otimes^2 B\rightarrow \mathfrak{gl}(A)$ are linear maps. Define two operations $\cdot$ and $[\cdot,\cdot,\cdot]$ on $A\oplus B$ by \eqref{com-asso-sum} and \eqref{3-lie-sum} respectively. Then $(A\oplus B,\cdot,[\cdot,\cdot,\cdot])$ is an  admissible transposed Poisson 3-Lie algebra if and only if $(A,B,\rho_A,\mu_A,\rho_B,\mu_B)$ is a matched pair of  admissible transposed Poisson 3-Lie algebras. 
		%We denote it  by $A\Join _{\mu_A,\rho_A}^{\mu_B,\rho_B}B$.
	\end{thm}
	\begin{proof} 
		By Proposition \ref{pro:direct-3-Lie} and Proposition \ref{pro:direct-comm-asso},  $(A\oplus B,\cdot,[\cdot,\cdot,\cdot])$ is a admissible transposed Poisson 3-Lie algebra if and only if for all $x_i\in A, y_i\in B, 1\leq i\leq 4$,
		\begin{eqnarray*}
			&&\underbrace{(x_4+y_4)\cdot[x_1+y_1,x_2+y_2,x_3+y_3]}_{(C1)}=\underbrace{[(x_4+y_4)\cdot(x_1+y_1),x_2+y_2,x_3+y_3]}_{(C2)}=0.
		\end{eqnarray*}	
		By \eqref{com-asso-sum} and \eqref{3-lie-sum}, we have
		\begin{eqnarray*}
			(C1)&=&(x_4+y_4)\cdot\big([x_1,x_2,x_3]_A+\rho_B(y_1,y_2)x_3+\rho_B(y_3,y_1)x_2+\rho_B(y_2,y_3)x_1  \\
			&&+[y_1,y_2,y_3]_B+\rho_A(x_1,x_2)y_3+\rho_A(x_3,x_1)y_2+\rho_A(x_2,x_3)y_1\big)\\
			&=&\underbrace{x_4\cdot_A[x_1,x_2,x_3]_A}_{(a1)}+\underbrace{y_4\cdot_B[y_1,y_2,y_3]_B}_{(a2)}+\underbrace{\mu_B(y_4)[x_1,x_2,x_3]_A}_{(c1)}\\
			&&+\underbrace{\mu_B(y_4)\big(\rho_B(y_1,y_2)x_3\big)}_{(b1)}+\underbrace{\mu_B(y_4)\big(\rho_B(y_3,y_1)x_2\big)}_{(b2)}+\underbrace{\mu_B(y_4)\big(\rho_B(y_2,y_3)x_1\big)}_{(b3)}\\
			&&+\underbrace{\mu_B([y_1,y_2,y_3]_B)x_4}_{(b4)}+\underbrace{\mu_B\big(\rho_A(x_1,x_2)y_3\big)x_4}_{(e1)}+\underbrace{\mu_B\big(\rho_A(x_3,x_1)y_2\big)x_4}_{(e2)}\\
			&&+\underbrace{\mu_B\big(\rho_A(x_2,x_3)y_1\big)x_4}_{(e3)}+\underbrace{\mu_A(x_4)[y_1,y_2,y_3]_B}_{(d1)}+\underbrace{\mu_A(x_4)\big(\rho_A(x_1,x_2)y_3\big)}_{(b5)}\\
			&&+\underbrace{\mu_A(x_4)\big(\rho_A(x_3,x_1)y_2\big)}_{(b6)}+\underbrace{\mu_A(x_4)\big(\rho_A(x_2,x_3)y_1\big)}_{(b7)}+\underbrace{\mu_A([x_1,x_2,x_3]_A)y_4}_{(b8)}\\
			&&+\underbrace{\mu_A\big(\rho_B(y_1,y_2)x_3\big)y_4}_{(f1)}+\underbrace{\mu_A\big(\rho_B(y_3,y_1)x_2\big)y_4}_{(f2)}+\underbrace{\mu_A\big(\rho_B(y_2,y_3)x_1\big)y_4}_{(f3)},\\
			(C2)&=&[x_4\cdot_A x_1 + \mu_B(y_4)x_1+ \mu_B(y_1)x_4+y_4\cdot_B y_1+\mu_A(x_4)y_1+\mu_A(x_1)y_4,x_2+y_2,x_3+y_3]\\
			&=&\underbrace{[x_4\cdot_A x_1,x_2,x_3]_A}_{(a3)}+\underbrace{[\mu_B(y_4)x_1,x_2,x_3]_A}_{(g1)}+\underbrace{[\mu_B(y_1)x_4,x_2,x_3]_A}_{(g2)}\\
			&&+\underbrace{\rho_B(y_4\cdot_B y_1,y_2)x_3}_{(b9)}+\underbrace{\rho_B\big(\mu_A(x_4)y_1,y_2\big)x_3}_{(j1)}+\underbrace{\rho_B\big(\mu_A(x_1)y_4,y_2\big)x_3}_{(j3)}  \\
			&&+\underbrace{\rho_B(y_3,y_4\cdot_B y_1)x_2}_{(b10)}+\underbrace{\rho_B\big(y_3,\mu_A(x_4)y_1\big)x_2}_{(j5)}+\underbrace{\rho_B\big(y_3,\mu_A(x_1)y_4\big)x_2}_{(j7)}\\
			&&+\underbrace{\rho_B(y_2,y_3)(x_4\cdot_A x_1)}_{(d2)} + \underbrace{\rho_B(y_2,y_3)\mu_B(y_4)x_1}_{(b11)}+ \underbrace{\rho_B(y_2,y_3)\mu_B(y_1)x_4}_{(b12)}\\
			&&+\underbrace{[y_4\cdot y_1,y_2,y_3]_B}_{(a4)}+\underbrace{[\mu_A(x_4)y_1,y_2,y_3]_B}_{(h1)}+\underbrace{[\mu_A(x_1)y_4,y_2,y_3]_B}_{(h2)}\\
			&&+\underbrace{\rho_A(x_4\cdot_A x_1,x_2)y_3}_{(b13)}+\underbrace{\rho_A\big(\mu_B(y_4)x_1,x_2\big)y_3}_{(j8)}+\underbrace{\rho_A\big(\mu_B(y_1)x_4,x_2\big)y_3}_{(j6)}\\
			&&+\underbrace{\rho_A(x_3,x_4\cdot_A x_1)y_2}_{(b14)}+\underbrace{\rho_A\big(x_3,\mu_B(y_4)x_1\big)y_2}_{(j4)}+\underbrace{\rho_A\big(x_3,\mu_B(y_1)x_4\big)y_2}_{(j2)}\\
			&&+\underbrace{\rho_A(x_2,x_3)(y_4\cdot_B y_1)}_{(c2)}+\underbrace{\rho_A(x_2,x_3)\mu_A(x_4)y_1}_{(b15)}+\underbrace{\rho_A(x_2,x_3)\mu_A(x_1)y_4}_{(b16)}.
		\end{eqnarray*}
		By \eqref{eq:trans-poisson-and-poisson}, we have $$(a1)=(a2)=(a3)=(a4)=0.$$
		By \eqref{eq:trans-poisson-and-poisson-rep-prop}, we have 
		\begin{align*}
			&(b1)=(b2)=(b3)=(b4)=(b5)=(b6)=(b7)=(b8)=(b9)\\
			&=(b10)=(b11)=(b12)=(b13)=(b14)=(b15)=(b16)=0.
		\end{align*}
		Moreover, we have the following equivalences
		\begin{eqnarray*}
			(c1)=(c2)=0 &\Leftrightarrow& \eqref{eq:admiss-match-pair-1};\\
			(d1)=(d2)=0 &\Leftrightarrow& \eqref{eq:admiss-match-pair-2};\\
			(e1)=(e2)=(e3)=0 &\Leftrightarrow& \eqref{eq:admiss-match-pair-3};\\
			(f1)=(f2)=(f3)=0 &\Leftrightarrow& \eqref{eq:admiss-match-pair-4};\\
			(g1)=(g2)=0 &\Leftrightarrow& \eqref{eq:admiss-match-pair-5};\\
			(h1)=(h2)=0 &\Leftrightarrow& \eqref{eq:admiss-match-pair-6};\\
			(j1)+(j2)=(j3)+(j4)=(j5)+(j6)=(j7)+(j8)=0 &\Leftrightarrow& \eqref{eq:admiss-match-pair-7}.
		\end{eqnarray*}
		Therefore, $(A,B,\rho_A,\mu_A,\rho_B,\mu_B)$ is a matched pair of admissible transposed Poisson 3-Lie algebras if and only if $(A\oplus B,\cdot,[\cdot,\cdot,\cdot])$ is an admissible transposed  Poisson 3-Lie algebra. The proof is completed.
	\end{proof}
	\begin{cor} \label{cor:adm-m}
		Let $(A,B,\rho_A,\rho_B)$ be a matched pair of 3-Lie algebras and $(A,B,\mu_A,\mu_B)$ be a matched pair of commutative associative algebras. Then $(A,B,\rho_A,\mu_A,\rho_B,\mu_B)$ is  both  a matched pair of Poisson 3-Lie algebras and a matched pair of transposed Poisson 3-Lie algebras if and only if it is a matched pair of admissible transposed Poisson 3-Lie algebras.
	\end{cor}
	\subsection{Double construction admissible transposed  Poisson 3-Lie bialgebras}
	\begin{defi}
		An \textbf{admissible transposed Poisson 3-Lie coalgebra} is a triple $(A,\Delta,\delta)$, where $(A,\delta)$ is a 3-Lie coalgebra and $(A,\Delta)$ is a cocommutative coassociative coalgebra that satisfies
		\begin{equation}\label{eq:admissible-Poisson-coalg}
			(\Delta\otimes 1\otimes 1)\delta(x)=(1\otimes \delta)\Delta(x)=0,\quad\forall~ x\in A.
		\end{equation}
	\end{defi}
	\begin{pro}
		Let $A$ be a vector space, $(A,\delta)$ be a 3-Lie coalgebra and $(A,\Delta)$ be a cocommutative coassociative coalgebra. Then $(A,\Delta,\delta)$ is an admissible transposed Poisson 3-Lie coalgebra if and only if $(A^*,\Delta^*,\delta^*)$ is an admissible transposed Poisson 3-Lie algebra.
	\end{pro}
	\begin{proof}
		By the proof of Proposition \ref{pro:poisson-3-lie-coal},  we only need to prove that $\Delta^*$ and $\delta^*$ satisfy \eqref{eq:trans-poisson-and-poisson}. For all $x\in A, a^*,b^*,c^*,d^* \in A^*$, we have
		\begin{eqnarray*}
			\langle (\Delta\otimes 1\otimes 1)\delta(x), a^*\otimes b^*\otimes c^*\otimes d^* \rangle = \langle x, \delta^*\big(\Delta^*(a^*\otimes b^*)\otimes c^*\otimes d^*\big) \rangle,
		\end{eqnarray*}
		and
		\begin{eqnarray*}
			\langle (1\otimes \delta)\Delta(x), a^*\otimes b^*\otimes c^*\otimes d^* \rangle =
			\langle x, \Delta^*\big(a^*\otimes \delta^*(b^*\otimes c^*\otimes d^*)\big)  \rangle.
		\end{eqnarray*}
		Thus \eqref{eq:admissible-Poisson-coalg} holds if and only if
		\begin{eqnarray*}
			\delta^*\big(\Delta^*(a^*\otimes b^*)\otimes c^*\otimes d^*\big)=\Delta^*\big(a^*\otimes \delta^*(b^*\otimes c^*\otimes d^*)\big) =0.
		\end{eqnarray*}
		Therefore, $(A,\Delta,\delta)$ is an admissible transposed Poisson 3-Lie coalgebra if and only if $(A^*,\Delta^*,\delta^*)$ is an admissible transposed Poisson 3-Lie algebra. The proof is completed.
	\end{proof}
	\begin{defi}
		Let $(A,\cdot, [\cdot,\cdot,\cdot])$ be an admissible transposed  Poisson 3-Lie algebra and $(A,\Delta,\delta)$ be an admissible transposed  Poisson 3-Lie coalgebra. If $(A,[\cdot,\cdot,\cdot],\delta)$ is a double construction 3-Lie bialgebra, $(A,\cdot,\Delta)$ is a commutative and cocommutative infinitesimal bialgebra satisfying:
		\begin{eqnarray}
			&&\Delta([x,y,z])=0,\label{eq:admi-bialg-1}\\
			&&\delta(x\cdot y)=0,\label{eq:admi-bialg-2}\\
			&&( \mathrm{ad}_{y,z}\otimes 1)\Delta(x)=0,\label{eq:admi-bialg-3}\\
			&&\big(1\otimes1\otimes\mathcal{L}(y)\big)\delta(x)=0,\label{eq:admi-bialg-4}\\
			&&( 1\otimes\mathrm{ad}_{y,z} )\Delta(x)=0,\label{eq:admi-bialg-5}\\
			&&\big(\mathcal{L}(y)\otimes1\otimes1\big)\delta(x)=0, \quad\forall~ x,y,z\in A, \label{eq:admi-bialg-6}
		\end{eqnarray}
		then we call $(A, \cdot,[\cdot,\cdot,\cdot],\Delta,\delta)$ a \textbf{double construction admissible transposed  Poisson 3-Lie bialgebra}.
	\end{defi}
	Using the same method in Lemma \ref{equival} we have the following conclusion.
	\begin{lem}\label{lam:admis-matched-pair}
		Let $(A,\cdot, [\cdot,\cdot,\cdot])$ and $(A^*,\circ,[\cdot,\cdot,\cdot]^*)$ be two admissible transposed  Poisson 3-Lie algebras. Let $(A^*,\mathrm{ad}^*,-\mathcal{L}^*)$ and $(A,\mathfrak{ad}^*,-L^*)$ be the coadjoint representation of $(A,\cdot,[\cdot,\cdot,\cdot])$ and $(A^*,\circ,[\cdot,\cdot,\cdot]^*)$, respectively. Then $(A, A^*, \mathrm{ad}^*,-\mathcal{L}^*,\mathfrak{ad}^*,-L^*)$ is a matched pair of admissible transposed  Poisson 3-Lie algebras if and only if \eqref{eq:admiss-match-pair-1} $\sim$ \eqref{eq:admiss-match-pair-6} hold for $\rho_A=\mathrm{ad}^*,\mu_A=-\mathcal{L}^*,\rho_B=\mathfrak{ad}^*,\mu_B=-L^*$.
	\end{lem}
	\begin{thm} \label{thm:adm-m-d}
		Let $(A,\cdot,[\cdot,\cdot,\cdot])$ be an admissible transposed Poisson 3-Lie algebra and  $(A,\Delta,\delta)$ be an  admissible transposed Poisson 3-Lie coalgebra.  Then $(A,A^*,\mathrm{ad}^*,-\mathcal{L}^*,\mathfrak{ad}^*,-L^*)$ is a matched pair of $(A,\cdot,[\cdot,\cdot,\cdot])$ if and only if $(A,\cdot, [\cdot,\cdot,\cdot],\Delta,\delta)$ is a double construction admissible transposed  Poisson 3-Lie bialgebra.
	\end{thm}
	\begin{proof}
		Using the structure constants in the proof of Lemma \ref{equival}, we can prove that 
		$$\eqref{eq:admiss-match-pair-1} \Leftrightarrow \eqref{eq:admi-bialg-1},~ \eqref{eq:admiss-match-pair-2} \Leftrightarrow \eqref{eq:admi-bialg-2},~ \eqref{eq:admiss-match-pair-3} \Leftrightarrow \eqref{eq:admi-bialg-3},~ \eqref{eq:admiss-match-pair-4} \Leftrightarrow \eqref{eq:admi-bialg-4},~ \eqref{eq:admiss-match-pair-5} \Leftrightarrow \eqref{eq:admi-bialg-5},~ \eqref{eq:admiss-match-pair-6} \Leftrightarrow \eqref{eq:admi-bialg-6}.$$ By Lemma \ref{lam:admis-matched-pair}, $(A,\cdot, [\cdot,\cdot,\cdot],\Delta,\delta)$ is a double construction  admissible transposed Poisson 3-Lie bialgebra if and only if  $(A,A^*,\mathrm{ad}^*,-\mathcal{L}^*,\mathfrak{ad}^*,-L^*)$ is a matched pair of Poisson 3-Lie algebras. The proof is completed. 
	\end{proof}
	\begin{cor} \label{cor:ad-d-bi}
		A double construction  admissible transposed Poisson 3-Lie bialgebra is  a double construction  Poisson 3-Lie bialgebra.
	\end{cor}
	\begin{ex} \label{ex:no-bi}
		Let $A$  be a 4-dimensional vector space with a basis $\{e_1,e_2,e_3,e_4\}$.   
		If we	define the  nonzero operations $\cdot$ and $[\cdot,\cdot,\cdot]$ on $A$ by
		$$[e_2,e_3,e_4]=e_1, \quad e_2\cdot e_2=e_1,$$and define the linear maps $\delta: A \to A \otimes A \otimes A$ and $\Delta:A\rightarrow  A\otimes A$ by $$\delta(e_2) = e_1 \wedge e_3 \wedge e_4, \quad \Delta(e_2)=e_1\otimes e_1,$$ then $(A, \cdot,[\cdot,\cdot,\cdot],\Delta,\delta)$ is a double construction admissible transposed  Poisson 3-Lie bialgebra.
	\end{ex}
	\begin{proof}
		By \cite[Theorem 4.8]{C. Du}, we have that $(A,[\cdot,\cdot,\cdot],\delta)$ is a double construction 3-Lie bialgebra. Moreover, $(A,\cdot)$ is a commutative associative algebra and $\Delta$ satisfies \eqref{eq:com-asso-coal} and \eqref{eq:infi-bialg}, that is, $(A,\cdot,\Delta)$ is a commutative and cocommutative infinitesimal bialgebra. 
		
		Furthermore, we can prove that $(A,\cdot,[\cdot,\cdot,\cdot])$ is an admissible transposed  Poisson 3-Lie algebra, and $\Delta,\delta$ satisfy \eqref{eq:admissible-Poisson-coalg} and \eqref{eq:admi-bialg-1} $\sim$ \eqref{eq:admi-bialg-6}, that is, $(A, \cdot,[\cdot,\cdot,\cdot],\Delta,\delta)$ is a double construction  admissible transposed Poisson 3-Lie bialgebra. 
	\end{proof}
	\begin{rmk} 
		Let $\{e^*_1,e^*_2,e^*_3,e^*_4\}$ be the dual basis of $A$. By the above double construction admissible transposed  Poisson 3-Lie bialgebras, we can give the admissible transposed  Poisson 3-Lie algebra $(A^*, \circ,[\cdot,\cdot,\cdot]^*)$, where 
		$$[e^*_1,e^*_3,e^*_4]^*=e^*_2, \quad e_1^*\circ e_1^*=e_2^*.$$ 
		Therefore, $(A\oplus A^*,\cdot_{A\oplus A^*},[\cdot,\cdot,\cdot]_{A\oplus A^*})$ is an admissible transposed  Poisson 3-Lie algebra,  where  $\cdot_{A\oplus A^*}$ and $[\cdot,\cdot,\cdot]_{A\oplus A^*}$  are defined by \eqref{Manin-1} and \eqref{Manin-2}, respectively.    
	\end{rmk}
	
	\section{Closing discussions} \label{sec:7}
	
	We close this paper by proposing some questions to be considered in the future projects:
	\begin{enumerate}
		\item In \cite{G. Liu}, bialgebra theory for transposed Poisson algebras uses bialgebra structures related to Manin triples of Lie algebras with respect to commutative 2-cocycles, which are closely linked to anti-pre-Lie algebras. Can we define (transposed) Poisson $3$-Lie bialgebras using the same approach?  
		\item Poisson bialgebras can be constructed from a combination of the classical Yang-Baxter equation and the associative Yang-Baxter equation. Can we construct Poisson 3-Lie bialgebras from the 3-Lie classical Yang-Baxter equation given in \cite{Bai2019}?  
		\item What is the bialgebra theory for (transposed) Poisson $n$-Lie $(n>3)$ algebras? Can we derive (transposed) Poisson $n$-Lie $(n\geq 3)$ algebras from (transposed) Poisson bialgebras as described in \cite{X. Ni} and \cite{G. Liu}?  
	\end{enumerate}

	\noindent{\bf Acknowledgements:}
	This work is supported by the Natural Science Foundation of Zhejiang Province (Grant No. LZ25A010002).

\end{document}